\documentclass[reqno]{amsart}
\usepackage[utf8]{inputenc}
\usepackage[english]{babel}

\usepackage{amsfonts,amsthm,amssymb,mathtools}
\usepackage[alphabetic,abbrev]{amsrefs}
\numberwithin{equation}{section}

\usepackage[unicode]{hyperref}
\usepackage[capitalise,noabbrev]{cleveref} 
\crefname{equation}{}{}
\Crefname{equation}{}{}

\newtheorem{theorem}{Theorem}[section]
\newtheorem{lemma}[theorem]{Lemma}
\newtheorem{proposition}[theorem]{Proposition}
\newtheorem{corollary}[theorem]{Corollary}

\theoremstyle{definition}
\newtheorem{definition}[theorem]{Definition}

\theoremstyle{remark}
\newtheorem{remark}[theorem]{Remark}

\newcommand{\N}{\mathbb{N}}
\newcommand{\Z}{\mathbb{Z}}

\newcommand{\R}{\mathbb{R}}
\newcommand{\C}{\mathbb{C}}

\newcommand{\g}{\mathfrak{g}}
\newcommand{\so}{\mathfrak{so}}
\newcommand{\su}{\mathfrak{su}}

\renewcommand{\S}{\mathcal{S}}

\DeclarePairedDelimiter\norm{\lVert}{\rVert}

\DeclareMathOperator{\supp}{supp}
\DeclareMathOperator{\id}{id}
\DeclareMathOperator{\ran}{ran}

\DeclareMathOperator{\diag}{diag}
\DeclareMathOperator{\tr}{tr}
\DeclareMathOperator{\sgn}{sgn}

\newcommand{\sloc}{\mathrm{sloc}}

\newcommand{\eigv}[3]{\lambda_{\mathbf{#1}}^{\mathbf{#2},\mathbf{#3}}}
\newcommand{\eigvp}[2]{\lambda_{\mathbf{#1}}^{#2}}

\AtBeginDocument{%
   \def\MR#1{}
}

\begin{document}

\title[Spectral multipliers on Métivier groups]{Spectral multipliers on Métivier groups}
\author{Lars Niedorf}
\address{Department of Mathematics, University of Wisconsin-Madison,
480 Lincoln Dr, Madison, WI-53706, USA}
\email{niedorf@wisc.edu}
\date{December 12, 2024}
\thanks{The author gratefully acknowledges the support by the Deutsche Forschungsgemeinschaft (DFG) through grant MU 761/12-1.}

\begin{abstract}
We prove an $L^p$-spectral multiplier theorem under the sharp regularity condition $s > d\left|1/p - 1/2\right|$ for sub-Laplacians on Métivier groups. The proof is based on a restriction type estimate which, at first sight, seems to be suboptimal for proving sharp spectral multiplier results, but turns out to be surprisingly effective. This is achieved by exploiting the structural property that for any Métivier group the first layer of any stratification of its Lie algebra is typically much larger than the second layer, a phenomenon closely related to Radon--Hurwitz numbers.
\end{abstract}


\maketitle

\section{Introduction}

\subsection{Spectral multipliers} The celebrated Mikhlin--Hörmander theorem \cite{Hoe60} states that the Fourier multiplier operator $T$ given by
\[
\widehat{Tf}(\xi)=F(\xi)\hat f(\xi),\quad f\in \mathcal S(\R^n),\xi\in\R^n
\]
is of weak type $(1,1)$ and extends to a bounded operator on $L^p(\R^n)$ for all $p\in (1,\infty)$ whenever the Fourier multiplier $F:\R^n\to\C$ satisfies
\[
\|F\|_{L^2_{s,\sloc}(\R^n)} := \sup_{t>0} {\norm{F(t\,\cdot\,)\eta }_{L^2_s(\R^n)}} < \infty\quad \text{for some } s>n/2.
\]
Here $\eta:\R^n\to\C$ denotes some nonzero smooth function which is compactly supported in an annulus away from the origin, and $L^2_s(\R^n)\subseteq L^2(\R^n)$ is the Sobolev space of fractional order $s\ge 0$. Note that different choices of $\eta$ lead to equivalent norms. The regularity condition $s>n/2$ is sharp and cannot be decreased, see for example \cite{SiWr01}.

The Mikhlin--Hörmander theorem implies the following spectral multiplier result for the Laplacian $-\Delta=-(\partial_{x_1}^2+\dots+\partial_{x_n}^2)$ on $\R^n$. Given a Borel measurable function $F:\R\to\C$, which we call a \textit{spectral multiplier} in the following, the operator $F(-\Delta)$ defined by functional calculus corresponds to a radial Fourier multiplier operator since
\[
(F(-\Delta)f)^\wedge(\xi)=F(|\xi|^2)\hat f(\xi).
\]
Thus, the Mikhlin--Hörmander theorem implies that $F(-\Delta)$ is bounded as an operator on $L^p(\R^n)$ for all $p\in (1,\infty)$ and is of weak type $(1,1)$ whenever
\[
\|F\|_{L^2_{s,\sloc}} := \sup_{t>0} {\norm{F(t\,\cdot\,)\eta }_{L^2_s(\R)}} < \infty\quad \text{for some } s>n/2.
\]

\subsection{Sub-Laplacians}

Let $G$ be a stratified Lie group and $\g=\g_1\oplus\dots\oplus\g_r$ be the stratification of the Lie algebra of $G$, that is, $[\g_i,\g_j]\subseteq \g_{i+j}$ for all $i,j\in \N\setminus\{0\}$ (with the understanding that $\g_m=\{0\}$ for $m>r$), and $\g_1$ generates $\g$ as a Lie algebra. Given a basis $X_1,\dots,X_{d_1}$ of the first layer $\g_1$, we can identify it with left-invariant vector fields on $G$ via the Lie derivative, and consider the associated left-invariant sub-Laplacian~$L$, which is given by
\[
L=-(X_1^2+\dots+X_{d_1}^2).
\]
Let $G$ be equipped with a left-invariant Haar measure and $L^p(G)$ be the associated Lebesgue space. The sub-Laplacian $L$ is positive (that is, $(Lf,f)_{L^2(G)}\ge 0$ for all functions $f$ in the domain of $L$, which is given by $\{f\in L^2(G):Lf\in L^2(G) \}$) and self-adjoint on $L^2(G)$. Hence, one can define via functional calculus for every spectral multiplier $F:\R\to\C$ the operator
\[
F(L) = \int_0^\infty F(\lambda)\,dE(\lambda)\quad \text{on } L^2(G),
\]
where $E$ denotes the spectral measure of $L$. By the spectral theorem, the operator $F(L)$ is bounded on $L^2$ if and only if the multiplier $F$ is $E$-essentially bounded. Christ \cite{Ch91}, and Mauceri and Meda \cite{MaMe90} showed that $F(L)$ is of weak type $(1,1)$ and bounded on all $L^p$-spaces for $p\in(1,\infty)$ if
\[
\|F\|_{L^2_{s,\sloc}} < \infty\quad \text{for some } s>Q/2,
\]
where $Q=\dim \g_1+2\dim \g_2+\dots +r\dim \g_r$ is the \textit{homogeneous dimension} of~$G$. For sub-Laplacians on Heisenberg (type) groups, Müller and Stein \cite{MueSt94}, and Hebisch \cite{He93} showed that the threshold $s>Q/2$ can even be pushed down to $s>d/2$, with $d$ being the topological dimension of the underlying group. Since then, this result has been extended to other specific classes of two-step stratified Lie groups \cite{Ma12,MaMue14b,Ma15}, and also to related settings \cite{CoSi01,MaMue14a,CoKlSi11,AhCoMaMue20,DaMa20,DaMa22a,DaMa22b}.

Conversely, Martini, Müller, and Nicolussi Golo \cite{MaMueNi23} were able to show for a very general class of smooth second-order real differential operators associated with a bracket-generating sub-Riemannian structure on smooth $d$-dimensional manifolds that at least regularity of order $s\geq d\left|1/p-1/2\right|$ is necessary for having the spectral multiplier estimate
\[
\norm{F(L)}_{p\to p} \le C_{p,s} \|F\|_{L^\infty_{s,\sloc}}
\]
for every bounded Borel function $F$, where
\[
\|F\|_{L^\infty_{s,\sloc}} := \sup_{t>0} {\norm{F(t\,\cdot\,)\eta }_{L^\infty_s(\R^n)}}.
\]
In particular, their result implies that regularity of order $s\ge d/2$ is necessary to have boundedness on all $L^p$-spaces. On the other hand, suppose we have a Mikhlin--Hörmander type result of the form
\[
\| F(L) \|_{p\to p} \le C_{p,s} \| F\|_{L^2_{s,\sloc}}
\quad\text{for all } p\in (1,\infty) \text{ and } s>d/2.
\]
Then, replacing $\|\cdot\|_{L^2_{s,\sloc}}$ by the stronger norm $\| \cdot \|_{L^\infty_{s,\sloc}}$, we obtain
\[
\| F(L) \|_{p\to p} \le C_{p,s} \| F\|_{L^\infty_{s,\sloc}} \quad\text{for all } s>d\left|1/p-1/2\right|,
\]
and all $p\in (1,\infty)$, which follows by interpolation with the trivial $L^2$-$L^2$ estimate from the spectral theorem. In the range $1\le p\le 2(d+1)/(d+3)$ such \textit{$p$-specific} spectral multiplier estimates are known to hold for radial Fourier multipliers on $\R^n$ even with the weaker $\smash{L^2_{s,\sloc}}$-norm, see \cite{Ch85,Se86,LeRoSe14}. With the $\smash{L^2_{s,\sloc}}$-norm in place of the $\smash{L^\infty_{s,\sloc}}$-norm on the right-hand side, these $p$-specific spectral multiplier estimates yield in particular that the Bochner-Riesz means given by
\[
F(L)=(1-tL)^\delta_+,\quad t > 0
\]
are uniformly bounded on $L^p$ under the regularity condition $\delta>d\left|1/p-1/2\right|-1/2$, which is the same order of regularity as in the Bochner-Riesz conjecture, see \cite{Ta99}.

So far it is an open question if regularity of order $s>\max\{d\left|1/p-1/2\right|,1/2\}$ with respect to the $L^2_{s,\sloc}$-norm is generally sufficient to provide spectral multiplier estimates. In sub-elliptic settings, the only results in the literature seem to be those of \cite{ChOu16,Ni22} for Grushin operators and those of \cite{Ni24} for sub-Laplacians on Heisenberg type groups.

\subsection{Métivier groups versus Heisenberg type groups}


The purpose of this paper is to extend the spectral multiplier result for sub-Laplacians on Heisenberg type groups of \cite{Ni24} to the larger class of Métivier groups.


Before stating the main result, we need to introduce some notation. Let $G$ be a two-step stratified Lie group, that is, a connected, simply connected nilpotent Lie group whose Lie algebra $\g$ (which is the tangent space $T_eG$ at the identity $e$ of $G$) admits a decomposition $\g=\g_1\oplus\g_2$ into two non-trivial subspaces $\g_1,\g_2\subseteq\g$, where $[\g_1,\g_1]=\g_2$ and $\g_2\subseteq\g$ is contained in the center of $\g$. In the following, we refer to $\g_1$ and $\g_2$ as being the \textit{first} and \textit{second layer} of $\g$, respectively. Let
\begin{equation}\label{eq:dimensions}
d_1=\dim \g_1\ge 1,\quad d_2=\dim \g_2\ge 1\quad \text{and}\quad d=\dim \g.
\end{equation}
Given a basis $X_1,\dots,X_{d_1}$ of the first layer $\g_1$, we identify each element of the basis by a left-invariant vector field on $G$ via the Lie derivative, and consider the associated sub-Laplacian $L$, which is the second-order differential operator
\begin{equation}\label{def:sub-Laplacian}
L = -(X_1^2+\dots+X_{d_1}^2).
\end{equation}
Let $\g_2^*$ denote the dual of $\g_2$. Any element $\mu\in\g_2^*$ gives rise to a skew-symmetric bilinear form $\omega_\mu:\g_1\times\g_1\to\R$ via
\[
\omega_\mu(x,x') = \mu([x,x']),\quad x,x'\in\g_1.
\]
The group $G$ is called a \textit{Métivier group} if $\omega_\mu$ is non-degenerate for all $\mu\in \g_2^*\setminus\{0\}$, which means that $\mathfrak r_\mu=\{0\}$ for the radical
\[
\mathfrak r_\mu = \{x\in \g_1:\omega_\mu(x,x')=0\text{ for all } x'\in\g_1\}.
\]
We choose a basis $U_1,\dots,U_{d_2}$ of the second layer $\g_2$. Let $\langle\cdot,\cdot\rangle$ be the inner product rendering $X_1,\dots,X_{d_1},U_1,\dots,U_{d_2}$ an orthonormal basis of $\g$. The inner product $\langle\cdot,\cdot\rangle$ induces a norm on the dual $\g_2^*$ which we denote by $|\cdot|$. Let $J_\mu$ be the skew-symmetric endomorphism such that
\begin{equation}\label{eq:skew-form-ii}
\omega_\mu(x,x') =\langle J_\mu x,x'\rangle,\quad x,x'\in\g_1.
\end{equation}
Then $G$ is a Métivier group if and only if $J_\mu$ is invertible for all $\mu\in\g_2^*\setminus\{0\}$. We call $G$ a \textit{Heisenberg type group} if the endomorphisms $J_\mu$ are orthogonal for all $\mu\in \g_2^*$ of length 1, which means that
\[
J_\mu^2 = - |\mu|^2 \id_{\g_1}\quad\text{for all } \mu\in \g_2^*.
\]
The class of Heisenberg type groups is a strict subclass of that of Métivier groups, an example of a Métivier group that is not of Heisenberg type may be found for instance in the appendix of \cite{MueSe04}, see also \cref{sec:spectral-theory}, where we analyze the spectral properties of the matrices $J_\mu$.

Passing to the broader class of Métivier groups generally involves giving up on the rotation invariant structure of Heisenberg type groups. Specifically, the matrices $J_\mu$ are no longer orthogonal for $|\mu| = 1$, resulting in a significantly more complicated spectral decomposition into eigenspaces parameterized by $\mu$. Similarly, the spectral properties of the sub-Laplacian $L$ become considerably more challenging to analyze compared to the Heisenberg type case, see \cref{sec:spectral-theory} below.

\subsection{Statement of the main results}

The main result of this paper is the following $p$-specific spectral multiplier estimate for Métivier groups, with an accompanying result for Bochner--Riesz multipliers, where the order of regularity is the same as in the Bochner--Riesz conjecture. As before, let
\[
\|F\|_{L^2_{s,\sloc}} = \sup_{t>0} {\norm{F(t\,\cdot\,)\eta}_{L^2_s(\R)}}.
\]
Given $n\in\N\setminus\{0\}$, let $p_n=2(n+1)/(n+3)$ be the  \textit{Stein--Tomas exponent} on $\R^n$. We put $p_{d_1,d_2}=p_{d_2}$ for $(d_1,d_2)\notin\{ (8,6),(8,7)\}$, $p_{8,6}=17/12$, and $p_{8,7}=14/11$.

\begin{theorem}\label{thm:multiplier}
Let $G$ be a Métivier group of topological dimension $d$, and, as in \cref{def:sub-Laplacian}, let $L$ be a sub-Laplacian on $G$. If $1\le p \le p_{d_1,d_2}$ with dimensions $d_1,d_2$ as in \eqref{eq:dimensions}, the following statements hold:
\begin{enumerate}
\item If $p>1$ and if $F:\R\to \C$ is a bounded Borel function such that \label{thm:part-1}
\[
\|F\|_{L^2_{s,\sloc}} < \infty
\quad\text{for some } s > d \left(1/p - 1/2\right),
\]
then the operator $F(L)$ is bounded on $L^p(G)$, and
\[
\norm{F(L)}_{L^p\to L^p} \le C_{p,s} \|F\|_{L^2_{s,\sloc}}.
\]
\item For any $\delta> d \left(1/p - 1/2\right)- 1/2$, the Bochner--Riesz means $(1-tL)^\delta_+$, $t\ge 0$, are uniformly bounded on $L^p(G)$.\label{thm:part-2}
\end{enumerate}
\end{theorem}

By the result of \cite{MaMueNi23}, the threshold for $s$ is optimal up the endpoint and cannot be decreased. If $s>d/2$ in the first part of \cref{thm:multiplier}, the operator $F(L)$ is of weak type $(1,1)$, as stated in \cite{Ma12}, and this holds independently of the dimensions $d_1$ and $d_2$. It should be pointed out that the first part of the theorem only provides results for $d_2\ge 2$. Since the proof of \cref{thm:multiplier} relies on exploiting favorable restriction type estimates, the fact that we do not get any results for $d_2=1$ aligns well with the fact that restriction estimates are only available when $p=1$ for Heisenberg groups \cite{Mue90}, unless one passes to mixed $L^p$-norms.

Comparing this result with the spectral multiplier theorem of \cite{Ni24} for Heisenberg type groups, \cref{thm:multiplier} is a generalization of Theorem 1.1 of \cite{Ni24}, except for the cases $(d_1,d_2)=(8,6)$ and $(d_1,d_2)=(8,7)$. We give some comments on these special cases. Note that in these cases the thresholds for $p$ in \cref{thm:multiplier} are $p_{8,6}=17/12$ and $p_{8,7}=14/11$, and we get spectral multiplier estimates for any Métivier group, while \cite{Ni24} gives the same estimates with thresholds $p_{6}=14/9>17/12$ and $p_{7}=8/5>14/11$, but only under the additional assumption that $L$ is a sub-Laplacian on a Heisenberg type group. In principle, it is also possible to reproduce the results of \cite{Ni24} for the cases $(d_1,d_2)=(8,6)$ and $(d_1,d_2)=(8,7)$ using the methods of this article by additionally incorporating the second layer weighted Plancherel estimates of Martini and Müller from \cite{MaMue14b}, see \cref{sec:4-3}, and in particular \cref{rem:extension-H-type}. However, the advantage of the Heisenberg type situation in \cite{Ni24} is twofold. On the one hand, the proof of the restriction type estimates in \cite{Ni24} is based on the spectral projector bounds for twisted Laplacians of Koch and Ricci \cite{KoRi07}, while the more general Métivier group situation considered here requires to resort to spectral cluster estimates for the associated anisotropic twisted Laplacians in order to control the distribution of their eigenvalues, see \cref{sec:sketch-restriction}. On the other hand, \cite{Ni24} does not require the second layer weighted Plancherel estimates of \cite{MaMue14b}, since the restriction type estimates for Heisenberg type groups in \cite{Ni24} are stronger than those of \cref{thm:restriction-type}.

A natural question is whether $p$-specific spectral multiplier results with threshold $s>d\left|1/p-1/2 \right|$ hold beyond the class of Métivier groups. However, despite the fact that by \cite{Ni25R}, there is now a restriction type estimate that holds in great generality, the bottleneck of the present approach seems to be the availability of favorable weighted Plancherel estimates. The requirement of having these estimates appears to be a major obstacle when passing from Métivier groups to classes of two-step groups where the bilinear form $\omega_\mu$ from \cref{eq:skew-form-ii} may have a non-trivial radical, see also \cref{rem:1st-layer}.

\subsection{Structure of the paper}

In \cref{sec:sketch}, we give a brief sketch of the proofs of \cref{thm:multiplier} and compare the approach of the present paper with that of \cite{Ni24}. In \cref{sec:spectral-theory}, we analyze the spectral properties of the sub-Laplacian $L$. In \cref{sec:abstract}, we reduce the spectral multiplier estimates of \cref{thm:multiplier} to estimates for multipliers whose frequencies are supported on dyadic scales, which is a minor variation of a result of \cite{ChOuSiYa16}.

From \cref{sec:numerology} onward, we restrict to the setting of Métivier groups. In \cref{sec:numerology}, we will see that the dimension $d_1$ of the first layer of a given Métivier group is in general much larger than the dimension $d_2$ of its center, up to the exceptional cases $(d_1,d_2)\in\{(4,3),(8,6),(8,7)\}$. A first layer weighted Plancherel estimate is proven in \cref{sec:Plancherel}, and the final proof of \cref{thm:multiplier} is carried out in Sections \ref{sec:1st-layer} and \ref{sec:4-3}.

\subsection{Notation}

We let $\N=\{0,1,2,\dots\}$. The indicator function of a subset $A$ of some measurable space will be denoted by $\mathbf{1}_A$. We write $A\lesssim B$ if $A\le C B$ for a constant $C$. If $A\lesssim B$ and $B\lesssim A$, we write $A\sim B$. Given two suitable functions $f$ and $g$ on a two-step stratified Lie group $G$, let $f*g$ denote their group convolution given by
\[
f*g(x,u) = \int_{G} f(x',u')g\big((x',u')^{-1}(x,u)\big) \,d(x',u'),\quad (x,u)\in G,
\]
where $d(x',u')$ denotes the Lebesgue measure on $G$. The space of Schwartz functions on $\R^n$ will be denoted by $\mathcal S(\R^n)$. For $s\ge 0$ and $q\in [1,\infty]$, we denote by $L^q_s(\R)\subseteq L^q(\R)$ the Sobolev space of fractional order $s$. If $\chi:\R^n\to \C$ is smooth and compactly supported, we also refer to $\chi$ as a bump function.

\subsection{Acknowledgments}

I am deeply grateful to my advisor Detlef Müller for his continuous support and many fruitful discussions about the subject of this work. I wish to thank Alessio Martini for hosting me for one week at the mathematical department of the Politecnico di Torino, and especially for giving me valuable insights into the second layer weighted Plancherel estimate technique, which led to the development of \cref{sec:4-3}. I am also indebted to Andreas Seeger for pointing out the connection between the dimensions $d_1,d_2$ of Métivier groups and Radon--Hurwitz numbers.


\section{Sketch of the proof} \label{sec:sketch}

This paper builds on the methods developed in \cite{ChOuSiYa16,ChOu16,Ni22,Ni24}, which in turn originate from C.\ Fefferman's idea of using restriction estimates to prove $L^p$-boundedness for Bochner--Riesz means, see \cite{Fe73}. This idea has also proved valuable in several other settings, see for instance \cite{SeSo89,Mue89,GuHaSi13,ChOuSiYa16,SiYaYa14} and the references therein.

\subsection{The restriction type estimates}\label{sec:sketch-restriction}

The approach of \cite{Ni24} relies on suitable restriction type estimates, and we want to exploit the restriction type estimates of \cite{Ni25R} for this. We first state the restriction type estimates of \cite{Ni25R}.

Let $G$ be again an arbitrary two-step stratified Lie group with Lie algebra $\g$ and stratification $\g=\g_1\oplus\g_2$. Let $L=(-X_1^2+\dots X_{d_1}^2)$ be the sub-Laplacian associated with a basis of the first layer $\g_1$ and let $U_1,\dots,U_2$ be a basis of the second layer $\g_2$. The operators $L$ and $-iU_1,\dots,-iU_{d_2}$ admit a joint functional calculus, which allows us to introduce an additional truncation along the spectrum of the operator
\[
U=(-(U_1^2+\dots+U_{d_2}^2))^{1/2},
\]
which corresponds to a partial Laplacian on the second layer $\g_2$. This allows us to write any operator $F(L)$ as
\[
F(L) = \sum_{\ell \in\Z} F(L) \chi(2^\ell U),
\]
where $\chi\in C_c^\infty(0,\infty)$ and $\chi(2^\ell \cdot)$ is a partition of unity. 

Due to this additional truncation, which will be crucial in proving the spectral multiplier estimates of \cref{thm:multiplier}, we also refer to the corresponding restriction type estimates as \textit{truncated} restriction type estimates. A similar truncation is used in \cite{Ni24} for sub-Laplacians in Heisenberg type groups, and in \cite{Ni22} in the related setting of Grushin operators.

It should be pointed out that the method of \cite{ChOu16}, which \cite{Ni22} and \cite{Ni24} are based on, relies on \textit{weighted} restriction type estimates, which turned out to be applicable for Grushin operators, but unfortunately not for sub-Laplacians on Heisenberg type groups, see \cite[Section 8]{Ni24}.

The restriction type estimates of \cite{Ni25R} are stated in terms of the norms 
\[
\norm{F}_{M,2} = \bigg(\frac{1}{M} \sum_{K\in\Z} \,\sup _{\lambda \in [\frac{K-1}{M}, \frac{K}{M})}|F(\lambda)|^2\bigg)^{1 / 2},\quad M\in(0,\infty)
\]
which were introduced by Cowling and Sikora in \cite{CoSi01}. Given any Euclidean space of dimension $n\in\N\setminus\{0\}$, we denote by
\[
p_n := \frac{2(n+1)}{n+3}
\]
the \textit{Stein--Tomas exponent} associated with that space.

\begin{theorem}[Truncated restriction type estimates]\cite{Ni25R}\label{thm:restriction-type}
Let $G$ be a two-step stratified Lie group, and, as in \cref{def:sub-Laplacian}, let $L$ be a sub-Laplacian on $G$. Suppose that $1\le p\le \min\{p_{d_1},p_{d_2}\}$ with dimensions $d_1,d_2$ as in \eqref{eq:dimensions}. If $F:\R\to\C$ is a bounded Borel function supported in a compact subset $A\subseteq (0,\infty)$ and $\chi:(0,\infty)\to\C$ is a smooth function with compact support, then
\begin{equation}\label{eq:intro-restriction}
\norm{ F(L)\chi(2^\ell U) }_{p\to 2}
\le C_{A,p,\chi} 2^{-\ell d_2(\frac 1 p - \frac 1 2)} \|F\|_2^{1-\theta_p} \norm{F}_{2^{\ell},2}^{\theta_p} \quad \text{for all }\ell\in\Z,
\end{equation}
where $\theta_p\in [0,1]$ satisfies $1/p = (1-\theta_p) + \theta_p/\min\{p_{d_1},p_{d_2}\}$.
\end{theorem}

\begin{remark}\label{rem:ell_0}
Let again $\langle\cdot ,\cdot \rangle$ be the inner product rendering $X_1,\dots,X_{d_1},U_1,\dots,U_{d_2}$ an orthonormal basis of $\g$ and let $J_\mu$ be the matrix with $\langle J_\mu x, x'\rangle = \mu([x,x'])$ for all $x,x'\in \g_1$. Then, as shown in \cite[Remark 5.2]{Ni25R} there is a constant $\ell_0\in\N$ which depends only on the matrices $J_\mu$, the inner product $\langle\cdot,\cdot\rangle$, and the compact subset $A\subseteq (0,\infty)$ above such that
\[
F(L)\chi(2^\ell U)=0 \quad\text{for all } \ell<-\ell_0.
\]
So \cref{thm:restriction-type} should actually be read as a statement for all $\ell\in [-\ell_0,\infty)$. 
\end{remark}

Although the above theorem is valid for all two-step stratified Lie groups, the presence of the norm $\norm{F}_{2^{\ell},2}$ in \eqref{eq:intro-restriction} imposes a slight drawback. By (3.19) and (3.29) of \cite{CoSi01} (or alternatively Lemma 3.4 of \cite{ChHeSi16}), for $s>1/2$, the norm $\Vert\cdot\Vert_{M,2}$ can be estimated by
\begin{equation}\label{eq:norm}
\norm{F}_{L^2}\le \norm{F}_{M,2} \le C_s \big( \norm{F}_{L^2} + M^{-s} \norm{F}_{L^2_s} \big),
\end{equation}
which means that $\Vert\cdot\Vert_{M,2}$ is stronger than the $L^2$-norm. Ignoring for a moment the additional truncation, the restriction type estimate \eqref{eq:intro-restriction} with $\norm{F}_{2^{\ell},2}$ replaced by $\norm{F}_{L^2}$ would be equivalent (see for instance \cite[Proposition 4.1]{SiYaYa14}) to a restriction type estimate for the Strichartz projectors $\mathcal P_\lambda$, which are formally given by $\mathcal P_\lambda=\delta_\lambda(L)$, where $\delta_\lambda$ is the Dirac delta distribution at $\lambda\in\R$. 

However, the truncated restriction type estimates can still be used to prove spectral multiplier estimates for the whole class of Métivier groups, by exploiting the fact that the dimension $d_1$ of the first layer is in general much larger than the dimension $d_2$ of the second layer if $G$ is a Métivier group. This fact is closely related to a theorem by Adams \cite{Ad62} which states that the maximal number of linearly independent vector fields on a $(n-1)$-dimensional sphere is given by $\rho_{\mathrm{RH}}(n)-1$, where $\rho_{\mathrm{RH}}(n)$ denotes the $n$-th Radon–Hurwitz number.

\subsection{The spectral multiplier estimates}

The key idea of deriving spectral multiplier estimates from restriction (type) estimates, which goes back to C.\ Fefferman \cite{Fe73}, can be illustrated as follows. Let $G$ be a Métivier group, and, as in \cref{def:sub-Laplacian}, let $L$ be a sub-Laplacian on $G$. Suppose that we want to derive $L^p$-boundedness of the Bochner--Riesz mean $(1-L)_+^\delta$. If $F:\R\to\C$ is the multiplier given by
\[
F(\lambda)=(1-\lambda^2)^\delta_+,\quad \lambda\in\R,
\]
then $F(\sqrt L)=(1-L)_+^\delta$. We decompose $F$ as
\[
F = \sum_{\iota = 0}^\infty F^{(\iota)},
\]
where $|\tau|\sim 2^\iota$ whenever $\tau \in\supp \widehat{F^{(\iota)}}$ for $\iota\ge 1$. It suffices to show
\[
\norm{F^{(\iota)}(\sqrt L)}_{p\to p}\lesssim 2^{-\varepsilon\iota}
\quad\text{for some }\varepsilon>0.
\]
Let $\mathcal K^{(\iota)}$ be the convolution kernel associated with $F^{(\iota)}(\sqrt L)$, that is,
\[
F^{(\iota)}(\sqrt L) f = f * \mathcal K^{(\iota)} \quad\text{for } f\in\mathcal S(G),
\]
where $*$ denotes the group convolution. By the Fourier inversion formula, we have
\[
F^{(\iota)} ( \sqrt L) = \frac{1}{2\pi} \int_{|\tau|\sim 2^\iota} \chi(\tau/2^\iota) \hat F(\tau) \cos(\tau \sqrt L) \,d\tau,
\]
where $\chi:\R\to\C$ is some bump function. Since solutions of the wave equation associated with $L$ possess finite propagation speed (which is the reason for artificially introducing the square root of $L$) and since we have $|\tau|\sim 2^\iota$ in the integral above, the convolution kernel $\mathcal K^{(\iota)}$ is supported in the Euclidean ball $B_R\times B_{R^2}\subseteq \R^{d_1}\times\R^{d_2}$ with radius $R\sim 2^\iota$ centered at the origin (where we identify \[G\cong \g_1\oplus\g_2\cong \R^{d_1}\oplus \R^{d_2}=\R^d\] via the basis $X_1,\dots,X_{d_1}$ of $\g_1$ and some basis $U_1,\dots,U_{d_2}$ of $\g_2$).

Now suppose we had a restriction type estimate of the form
\begin{equation}\label{intro-restriction-1}
\norm{F^{(\iota)}(\sqrt L)f}_2  \lesssim \norm{F^{(\iota)}}_2 \norm{f}_p,
\end{equation}
which is, for instance, the case if $G$ is a Heisenberg type group (up to replacing the multiplier $F^{(\iota)}$ by some localized version $F^{(\iota)}\psi$), see \cite{LiZh11} and \cite{Ni24}. Then, for a function $f$ supported in a Euclidean ball of size $R \times R^2$ centered at the origin (which one can assume without loss of generality), Hölder's inequality together with the restriction type estimate \cref{intro-restriction-1} yields
\begin{equation}\label{eq:intro-hoelder-1}
\norm{F^{(\iota)}(L)f}_p \lesssim 
 R^{Q/q} \norm{F^{(\iota)}(L)f}_2
 \lesssim R^{Q/q} \norm{F^{(\iota)}}_2 \norm{f}_p,
\end{equation}
where $1/q=1/p-1/2$ and $Q$ is the homogeneous dimension of $G$, which is
\[
Q=d_1+2d_2.
\]
One can show that $F\in L^2_{\smash{Q/q+\varepsilon}}(\R)$ for some $\varepsilon>0$ if $\delta>Q/q-1/2$. Then, the last term of \eqref{eq:intro-hoelder-1} can be estimated via
\[
R^{Q/q} \norm{F^{(\iota)}}_2 \norm{f}_p
 \lesssim 2^{-\varepsilon\iota} \norm{F^{(\iota)}}_{L^2_{Q/q+\varepsilon}(\R)} \norm{f}_p.
\]
Note that this argument requires an order of regularity in terms of the homogeneous dimension $Q$. The method of \cite{Ni22,Ni24} refines the dyadic decomposition $F^{(\iota)}$ by an additional truncation along the spectrum of the operator
\[
U = (-(U_1^2+\dots+U_{d_2}^2))^{1/2}.
\]
More precisely, we localize $F^{(\iota)}$ in space by some suitable non-zero smooth function $\psi$ which is compactly supported away from the origin and decompose the operator $(F^{(\iota)}\psi)(\sqrt L)$ as
\begin{equation}\label{eq:intro-truncation}
(F^{(\iota)}\psi)(\sqrt L) = \sum_{\ell\in\Z} (F^{(\iota)}\psi)(\sqrt L) \chi(2^\ell U)= \sum_{\ell\in\Z} F^{(\iota)}_\ell(L,U),
\end{equation}
where the functions $\chi(2^\ell\,\cdot\,)$ shall form a dyadic partition of unity of $\R\setminus\{0\}$. (The decomposition in \cite{Ni22,Ni24} is slightly different since the operator $\chi(2^\ell U)$ is replaced by $\chi(2^{-\ell} L/U)$ there. However, since the spectral multiplier $F^{(\iota)}\psi$ is compactly supported away from the origin, both the truncation by $\chi(2^\ell U)$ and $\chi(2^{-\ell} L/U)$ amount to having a multiplier for $U$ supported around $2^{-\ell}$.)

The convolution kernel $\mathcal K^{(\iota)}_\ell$ of $F^{(\iota)}_\ell(L,U)$ can be written down explicitly in terms of the partial Fourier transform along the center and Laguerre functions, which is as follows. As we will show in \cref{sec:spectral-theory}, there exists non-empty, homogeneous Zariski-open subset $\g_{2,r}^*\subseteq\g_2^*$, numbers $N\in\N\setminus\{0\}$, $\mathbf r=(r_1,\dots,r_N)\in(\N\setminus\{0\})^N$, and a continuous function $\g_2^*\ni \mu \mapsto \mathbf{b}^\mu = (b^\mu_1,\dots,b^\mu_N)\in [0,\infty)^N$ such that the skew-symmetric matrices $J_\mu$ given by \cref{eq:skew-form-ii} admit the spectral decomposition
\begin{equation}\label{eq:J_mu-spectral}
-J_\mu^2 = \sum_{n=1}^N (b_n^\mu)^2 P_n^\mu \quad \text{for all } \mu \in \g_{2,r}^*,
\end{equation}
where $P_n^\mu$ is an orthogonal projections on $\g_1$ of rank $2r_j$ for all $\mu \in \g_{2,r}^*$, with pairwise orthogonal ranges. If $G$ is a Métivier group, then, for $x\in\g_1$, $u\in\g_2$,
\begin{align}
\mathcal K^{(\iota)}_\ell(x,u) = (2\pi)^{-d_2} \sum_{\mathbf{k}\in\N^N} & \int_{\g_{2,r}^*} (F^{(\iota)}\psi)\big(\sqrt{\eigvp{k}{\mu}}\big) \chi(2^\ell |\mu|) \notag \\
& \times \prod_{n=1}^N \varphi_{k_n}^{(b_n^\mu,r_n)}(R_\mu^{-1}P_n^\mu x)\, e^{i\langle \mu, u\rangle} \, d\mu,
\label{eq:intro-conv-kernel}
\end{align}
where $R_\mu\in O(d_1)$,
\[
\eigvp{k}{\mu} = \eigv{k}{b^\mu}{r} = \sum_{n=1}^N \left(2k_n+r_n\right)b_n^\mu,
\]
and $\varphi_{k_n}^{(b_n^\mu,r_n)}$ are some rescaled Laguerre functions of the form
\begin{equation}\label{eq:rescaled-Lag}
\varphi_{k_n}^{(b_n^\mu,r_n)}(z) = (b_n^\mu)^{r_n} \varphi_{k_n}((b_n^\mu)^{1/2}z),\quad z\in \R^{2r_n}.
\end{equation}
If $G$ is a Métivier group, then $b_n^\mu\neq 0$ for all $\mu\neq 0$ and all $n\in\{1,\dots,N\}$, whence
\[
b_n^\mu \sim 1 \quad \text{for } |\mu|=1.
\]
Since $\eigvp{k}{\mu}\sim 1$ and $|\mu|\sim 2^\ell$ on the support of the involved multipliers, and as the function $\mu\mapsto \textbf{b}^\mu$ is homogeneous of degree $1$, the sum in \eqref{eq:intro-conv-kernel} ranges in fact over all $\mathbf{k}\in\N^N$ satisfying
\[
|\textbf{k}|+1 \sim \eigvp{k}{\bar\mu} = |\mu|^{-1}\eigvp{k}{\mu} \sim |\mu|^{-1}\sim 2^\ell,\quad\text{where } \bar \mu = \frac{\mu}{|\mu|}.
\]
In particular, the sum of the dyadic decomposition \eqref{eq:intro-truncation} ranges only over $\ell\in [-\ell_0,\infty)$, where $\ell_0\in\N$ is some constant that depends only on the matrices $J_\mu$ and the inner product $\langle\cdot,\cdot\rangle$ on $\g$, with respect to which $X_1,\dots,X_{d_1},U_1,\dots,U_{d_2}$ is an orthonormal basis of $\g$.

A crucial observation in connection with the formula \cref{eq:intro-conv-kernel} is the following. The pointwise estimates from \cite[Equations~(1.1.44), (1.3.41), Lemma~1.5.3]{Th93} suggest that any of the rescaled Laguerre functions of the form \eqref{eq:rescaled-Lag} is essentially supported in a ball of radius comparable to
\[
(b_n^\mu)^{-1/2}(k_n+1)^{1/2} \sim |\mu|^{-1/2}(k_n+1)^{1/2} \lesssim 2^\ell
\]
centered at the origin, up to some exponentially decaying term. Hence we may think of the convolution kernel $\mathcal K^{(\iota)}_\ell$ in \eqref{eq:intro-conv-kernel} being supported in a Euclidean ball of dimensions $R_\ell \times R^2$ in place of $R\times R^2$, with $R_\ell \sim 2^\ell$. This leads us to distinguish between the case $\ell\le \iota$, where $R_\ell\lesssim 2^\iota \sim R$, and the case $\ell>\iota$. Using the truncated restriction type estimate of \cref{thm:restriction-type} and applying the Sobolev type estimate \eqref{eq:norm} for the norm $\|\cdot\|_{2^\ell,2}$, we obtain
\begin{align}
\norm{ F_\ell^{(\iota)}(L,U) }_{p\to 2}
& \lesssim 2^{-\ell d_2/q} \|(F^{(\iota)}\psi)(\sqrt{\cdot}\,)\|_2^{1-\theta_p} \|(F^{(\iota)}\psi)(\sqrt{\cdot}\,)\|_{2^{\ell},2}^{\theta_p} \notag \\
& \le 2^{-\ell d_2/q} \|(F^{(\iota)}\psi)(\sqrt{\cdot}\,)\|_{2^{\ell},2} \notag \\
& \lesssim_{\psi} 2^{-\ell d_2/q} \big( \norm{F^{(\iota)}}_{L^2} + 2^{-\ell s} \norm{F^{(\iota)}}_{L^2_s} \big),\label{eq:intro-restriction-ii}
\end{align}
where again $1/q=1/p-1/2$, and $s>1/2$.

For $\ell>\iota$, we argue as before in \cref{eq:intro-hoelder-1}, using the fact that $\mathcal K^{(\iota)}$ is supported in a ball of size $R\times R^2$ (which is however not quite true for the kernels $\smash{\mathcal K_\ell^{(\iota)}}$, but the information about the support at scale $R\times R^2$ can be preserved by introducing a cut-off function $\mathbf{1}_{CR\times CR^2}$ before doing the decomposition \eqref{eq:intro-truncation}). Similar to \eqref{eq:intro-hoelder-1}, using Hölder's inequality, but now in combination with the restriction type estimate \eqref{eq:intro-restriction-ii} yields
\begin{align}
\norm{\mathbf{1}_{CR\times CR^2} (F_\ell^{(\iota)}(L,U)f)}_p
 & \lesssim R^{Q/q} \norm{F_\ell^{(\iota)}(L,U)f}_2 \notag \\
 & \lesssim R^{Q/q} 2^{-\ell d_2/q} \big( \norm{F^{(\iota)}}_{L^2} + 2^{-\ell s} \norm{F^{(\iota)}}_{L^2_s} \big) \norm{f}_p.\label{eq:intro-hoelder-1a}
\end{align}
Summing over all $\ell > \iota$ shows that the additional factor $2^{-\ell s}$ completely makes up for the higher order of the Sobolev norm in the factor $\norm{F}_{L^2_s}$ since
\[
\sum_{\ell=\iota+1}^\infty 2^{-\ell d_2/q} \big( \norm{F^{(\iota)}}_{L^2} + 2^{-\ell s} \norm{F^{(\iota)}}_{L^2_s} \big) 
\sim 2^{-\iota d_2/q} \norm{F^{(\iota)}}_{L^2}.
\]
Hence, since $Q-d_2=d$, \eqref{eq:intro-hoelder-1a} yields
\begin{align*}
\sum_{\ell=\iota+1}^\infty \norm{\mathbf{1}_{CR\times CR^2} (F_\ell^{(\iota)}(L,U)f)}_p
& \lesssim R^{d/q} \norm{F^{(\iota)}}_{L^2} \norm{f}_p \\
& \sim \norm{F^{(\iota)}}_{L_{d/q}^2} \norm{f}_p,
\end{align*}
so that we end up with an order of regularity $d/q=d\left(1/p-1/2\right)$ instead of $Q/q=Q\left(1/p-1/2\right)$ with homogeneous dimension $Q$.

For $\ell\le \iota$, we decompose the support of $f$, which is of size $R \times R^2$, into a grid of balls of size $R_\ell \times R^2$, and decompose $f$ into a sum of functions $f_m^{(\ell)}$, each of which is supported in one of the balls of the grid. Since the convolution kernel $\mathcal K_\ell^{(\iota)}$ is essentially supported in a ball of size $R_\ell\times R^2$, we may think of the function $F_\ell^{(\iota)}(L)f_m^{(\ell)}$ being also supported in a ball of size comparable to $R_\ell\times R^2$. Using again Hölder's inequality and the restriction type estimate \eqref{eq:intro-restriction-ii} yields
\begin{align}
\norm{F_\ell^{(\iota)}(L,U)f_m^{(\ell)}}_p
& \lesssim 
 (R_\ell^{d_1} R^{2d_2})^{1/q} \norm{F_\ell^{(\iota)}(L,U)f_m^{(\ell)}}_2 \notag \\
& \lesssim (R_\ell^{d_1-d_2} R^{2d_2})^{1/q} \big( \norm{F^{(\iota)}}_{L^2} + 2^{-\ell s} \norm{F^{(\iota)}}_{L^2_s} \big) \norm{f_m^{(\ell)}}_p.
\label{eq:intro-hoelder-2}
\end{align}
Since $\ell\le \iota$, and
\[
2^{-\ell s} \norm{F^{(\iota)}}_{L^2_s}\sim 2^{(\iota-\ell)s}\norm{F^{(\iota)}}_{L^2},
\]
the last line of \eqref{eq:intro-hoelder-2} is comparable to $\norm{F^{(\iota)}}_{L^2} \norm{f_m^{(\ell)}}_p$ times
\[
(2^{\ell (d_1-d_2)} 2^{2\iota d_2})^{1/q} 2^{(\iota-\ell)s}
 = (2^{(\ell-\iota)(d_1-d_2-sq)} 2^{\iota d})^{1/q}.
\]
In a precise argument that we will see later, since
\[
2^{\iota d/q}\norm{F^{(\iota)}}_{L^2}\sim 2^{-\varepsilon\iota}\norm{F^{(\iota)}}_{L_{d/q+\varepsilon}^2},
\]
the argument of the proof works as long as we can bound the sum
\[
\sum_{\ell=-\ell_0}^\iota \big(2^{(\ell-\iota)(d_1-d_2-sq)/q-\iota\varepsilon /2}\big)^p.
\]
By first choosing $\varepsilon>0$ and then $s>1/2$ sufficiently close to $1/2$, this can be achieved as long as
\[
d_1-d_2-\frac q 2\ge 0.
\]
Since we have in particular assumed $1\le p\le p_{d_2}=2(d_2+1)/(d_2+3)$ in \cref{thm:multiplier}, we have $1/q=1/p-1/2\ge 1/(d_2+1)$ and thus
\begin{equation}\label{eq:intro-condition}
d_1 - d_2 - \frac q 2
\ge d_1 - \frac{3d_2}{2} - \frac{1}{2}.
\end{equation}
Since $G$ is assumed to be a Métivier group, the map $\g_2^*\to(\g_1/\R y)^*,\mu\mapsto \omega_\mu(\cdot,x')$ is injective for $x'\neq 0$, whence $d_1>d_2$. But as a matter of fact, $d_1$ is in general much larger than $d_2$, and the latter term of \eqref{eq:intro-condition} is non-negative for any Métivier group except for the cases where
\begin{equation}\label{eq:intro-exceptional}
(d_1,d_2)\in\{(4,3),(8,6),(8,7)\}.
\end{equation}

\subsection{On the exceptional cases}

In the present approach, the fact that the convolution kernel $\smash{\mathcal K_\ell^{(\iota)}}$ is essentially supported in a ball of size $R_\ell\times R^2$ is derived from a weighted Plancherel estimate of the form
\begin{equation}\label{eq:intro-Plancherel-1}
\int_G \big| |x|^\alpha \mathcal K_\ell^{(\iota)}(x,u) \big|^2 \,d(x,u) \lesssim_\alpha 2^{\ell(2\alpha-d_2)} \norm{F^{(\iota)}}_{L^2}^2,\quad\alpha\ge 0.
\end{equation}
The remaining cases \eqref{eq:intro-exceptional} could also be covered if we additionally had a weighted Plancherel estimate of the form 
\begin{equation}\label{eq:intro-Plancherel-2}
\int_G \big| |u|^\beta \mathcal K_\ell^{(\iota)}(x,u) \big|^2 \,d(x,u) \lesssim_\alpha 2^{\ell(2\beta-d_2)} \norm{F^{(\iota)}}_{L^2_\beta}^2,\quad\beta\ge 0,
\end{equation}
with weight $|u|^\beta$ on the second layer. The weighted Plancherel estimate \eqref{eq:intro-Plancherel-2} would in fact yield that we may think of $\smash{\mathcal K_\ell^{(\iota)}}$ as being essentially supported in an even smaller ball of size $R_\ell\times R_\ell R$ rather than $R_\ell\times R^2$, up to error terms. Such estimates with weight on the second layer are available for Heisenberg type groups \cite{He93}, the more general class of Heisenberg--Reiter (type) groups \cite{Ma15}, and for classes of two-step groups in low dimensions \cite{MaMue14b}. It is possible to prove a second layer weighted Plancherel estimate on arbitrary two-step stratified Lie groups for small $\beta\ge 0$ by homogeneity considerations \cite{MaMue16}, which we unfortunately cannot exploit in our approach, since it requires \eqref{eq:intro-Plancherel-1} and \eqref{eq:intro-Plancherel-2} for arbitrarily large powers $\alpha,\beta\ge 0$.

On the other hand, according to \cite[Proposition 20]{MaMue14b}, second layer weighted Plancherel estimates result from appropriate bounds for $\mu$-derivatives of both the eigenvalues $b_1^\mu,\dots,b_N^\mu$ and the associated spectral projections $P_1^\mu,\dots,P_N^\mu$ from the spectral decomposition from \cref{eq:J_mu-spectral}, that is,
\[
-J_\mu^2 = \sum_{n=1}^N (b_n^\mu)^2 P_n^\mu \quad \text{for all } \mu \in \g_{2,r}^*.
\]
As we will discuss in \cref{sec:4-3}, such bounds are not always satisfied, since the functions $\mu\mapsto b_n^\mu$ and $\mu\mapsto P_n^\mu$ can admit singularities on the Zariski-closed subset $\g_2^*\setminus\g_{2,r}^*$. Fortunately, in the case $(d_1,d_2)=(4,3)$, the set of those singularities, which is always a cone due to homogeneity, is either a plane or a line in $\g_2^*$, which allows us to prove a second layer weighted Plancherel estimate not with the complete weight $|u|^\beta$, but with a weight corresponding to a directional derivative in $\mu$ whose direction lies within the Zariski-closed subset.

In the higher-dimensional exceptional cases, however, the set of possible singularities may no longer be a linear subspace, which is why the low-dimensional approach for $(d_1,d_2) = (4,3)$ fails in higher dimensions.

\section{Spectral decompositions} \label{sec:spectral-theory}

Let $G$ be a two-step stratified Lie group. Then its Lie algebra $\g$ admits a decomposition $\g=\g_1\oplus\g_2$, where $[\g_1,\g_1]=\g_2$ and $\g_2\subseteq \g$ is contained in the center of~$\g$. If we identify $G$ with its Lie algebra $\g$ via exponential coordinates, the group multiplication is given by
\[
(x,u)(x',u')=\left(x+x',u+u'+\tfrac 1 2 [x,x']\right),\quad x,x'\in \g_1,u,u'\in \g_2.
\]
Furthermore, we identify $\g_1\cong\R^{d_1}$ and $\g_2\cong\R^{d_2}$ by choosing a basis $X_1,\dots,X_{d_1}$ of $\g_1$ and a basis $U_1,\dots,U_{d_2}$ of $\g_2$. Let $\langle \cdot,\cdot\rangle$ denote the inner product with respect to which these two bases become an orthonormal basis of $\g$.

Let $\g_2^*$ be the dual of $\g_2$. For any $\mu\in\g_2^*$, let $J_\mu$ be the endomorphism given by
\[
\langle J_\mu x,x' \rangle = \mu([x,x']),\quad x,x'\in\g_1,
\]
Note that $J_\mu$ is skew-adjoint and $-J_\mu^2=J_\mu^* J_\mu$ is self-adjoint and non-negative. The family of matrices $J_\mu$ together with the inner product $\langle\cdot,\cdot\rangle$ fully characterizes the Lie group $G$. The matrices $J_\mu$ admit a “simultaneous spectral decomposition” if $\mu$ ranges in a Zariski open subset of $\g_2^*$.

\begin{proposition}\label{prop:rotation}
There exist a non-empty, homogeneous Zariski-open subset $\g_{2,r}^*$ of $\g_2^*$, numbers $N\in\N\setminus\{0\}$, $r_0\in\N$, $\smash{\mathbf r=(r_1,\dots,r_N)\in(\N\setminus\{0\})^N}$, a function $\smash{\mu \mapsto \mathbf{b}^\mu = (b^\mu_1,\dots,b^\mu_N)\in [0,\infty)^N}$ on $\g_2^*$, functions $\mu\mapsto P_n^\mu$ on $\g_{2,r}^*$ with $P_n^\mu:\g_1\to\g_1$, $n\in\{1,\dots,N\}$, and a function $\mu\mapsto R_\mu\in O(d_1)$ on $\g_{2,r}^*$ such that
\begin{equation}\label{eq:spectral-decomp-0}
-J_\mu^2 = \sum_{n=1}^N (b_n^\mu)^2 P_n^\mu \quad \text{for all } \mu \in \g_{2,r}^*,
\end{equation}
with $P_n^\mu R_\mu=R_\mu P_n$, $J_\mu(\ran P_n^\mu)\subseteq \ran P_n^\mu$ for the range of $P_n^\mu$ for all $\mu \in \g_{2,r}^*$ and all $n\in\{0,\dots,N\}$, where $P_n$ denotes the projection from $\R^{d_1}=\R^{r_0}\oplus\R^{2r_1}\oplus \dots\oplus \R^{2r_N}$ onto the $n$-th layer, where
\begin{itemize}
\item[(i)] the functions $\mu \mapsto b^\mu_n$ are homogeneous of degree~$1$ and continuous on $\g_2^*$, real analytic on $\g_{2,r}^*$, and satisfy $b_n^\mu > 0$ for all $\mu \in \g_{2,r}^*$ and $n \in \{1,\dots,N\}$, and $b_n^\mu \neq b_{n'}^\mu$ if $n \neq n'$ for all $\mu \in \g_{2,r}^*$ and $n,n' \in \{1,\dots,N\}$,
\item[(ii)] the functions $\mu\mapsto P_n^\mu$ are (componentwise) real analytic on $\g_{2,r}^*$, homogeneous of degree $0$, and the maps $P_n^\mu$ are orthogonal projections on $\g_1$ of rank $2r_n$ for all $\mu \in \g_{2,r}^*$, with pairwise orthogonal ranges,
\item[(iii)] $\mu\mapsto R_\mu$ is a Borel measurable function on $\g_{2,r}^*$ which is homogeneous of degree~$0$, and there is a family $(U_\ell)_{\ell\in\N}$ of disjoint Euclidean open subsets $U_\ell\subseteq \g_{2,r}^*$ whose union is $\g_{2,r}^*$ up to a set of measure zero such that $\mu\mapsto R_\mu$ is (componentwise) real analytic on each $U_\ell$.
\end{itemize}
\end{proposition}

\begin{proof}
See \cite{Ni25R}. See also \cite[Lemma 5]{MaMue14b}.
\end{proof}

As usual, we will identify the bases $X_1,\dots,X_{d_1}$ and $U_1,\dots,U_{d_2}$ with smooth left-invariant vector fields via Lie derivative. The sub-Laplacian
\[
L=-(X_1^2+\dots+X_{d_1}^2)
\]
and the vector $\textbf{U}=(-iU_1,\dots,-iU_{d_2})$ of differential operators admit a joint functional calculus \cite{Ma11}. For suitable functions $F:\R\times \R^{d_2}\to\C$, the operator $F(L,\mathbf U)$ possesses a (group) convolution kernel $\mathcal K_{F(L,\mathbf U)}$, that is,
\[
F(L,\mathbf U) f = f * \mathcal K_{F(L,\mathbf U)} \quad \text{for all } f\in \S(G).
\]
The convolution kernel $\mathcal K_{F(L,\mathbf U)}$ can be explicitly written down in terms of the Fourier transform and rescaled Laguerre functions.

For $\lambda>0$ and $m\in\N\setminus\{0\}$, let $\varphi_k^{(\lambda,m)}$ denote the $\lambda$-rescaled Laguerre function given by
\begin{equation}\label{eq:Laguerre}
\varphi_k^{(\lambda,m)}(z) = \lambda^m L_k^{m-1}\big(\tfrac 1 2 \lambda |z|^2\big) \,e^{-\frac 1 4 \lambda |z|^2},\quad z\in\R^{2m},
\end{equation}
where $L_k^{m-1}$ is the $k$-th Laguerre polynomial of type $m-1$. We use the notation of \cref{prop:rotation} in the next proposition.

\begin{proposition}\label{prop:conv-kernel}
If $F:\R\times \R^{d_2}\to\C$ is a Schwartz function, then $F(L,\mathbf U)$ possesses a convolution kernel $\mathcal K_{F(L,\mathbf U)}\in \S(G)$. For $x\in\g_1$ and $u\in\g_2^*$, we have
\begin{align}\label{eq:conv-kernel}
\mathcal K_{F(L,\mathbf U)}(x,u) = & (2\pi)^{-r_0-d_2} \int_{\g_{2,r}^*}   \int_{\R^{r_0}} \sum_{\mathbf{k}\in\N^N} F(|\tau|^2+\eigvp{k}{\mu},\mu) \notag \\
& \times \bigg[\prod_{n=1}^N \varphi_{k_n}^{(b_n^\mu,r_n)}(R_\mu^{-1} P_n^\mu x)\bigg] e^{i\langle \tau, R_\mu^{-1} P_0^\mu x\rangle} \, e^{i\langle \mu, u\rangle} \, d\tau  \, d\mu,
\end{align}
where $P_0^\mu = \id_{\g_1} - \left(P_1^\mu+\dots+P_N^\mu\right)$ and
\[
\eigvp{k}{\mu} = \eigv{k}{b^\mu}{r} = \sum_{n=1}^N \left(2k_n+r_n\right)b_n^\mu.
\]
\end{proposition}

\begin{proof}
See \cite{Ni25R}, or \cite[Proposition 6]{MaMue14b}).
\end{proof}

\begin{remark}\label{rem:conv-kernel}
In \cite[Proposition 6]{MaMue14b}, the convolution kernel $\mathcal K_{F(L,\mathbf U)}$ is presented in the form
\begin{equation}\label{eq:conv-kernel-alt}
\mathcal K_{F(L,\mathbf U)}(x,u) = \frac{2^{|r|_1}}{(2\pi)^{\dim G}}
 \int_{\g_{2,r}^*} \int_{\g_1} V_F(\xi,\mu) \, e^{i \langle \xi, x \rangle} \, e^{i \langle \mu, u \rangle} \,d\xi \,d\mu,
\end{equation}
where
\begin{equation}\label{eq:V-tilde}
V_F(\xi,\mu) = \sum_{k \in \N^N} F(|P^\mu_0 \xi|^2+\eigvp{k}{\mu},\mu) \prod_{n=1}^N \mathcal L_{k_n}^{(r_n-1)}(|P^\mu_n \xi|^2 /b^\mu_n)
\end{equation}
with Laguerre functions $\mathcal L_{k_n}^{(r_n-1)}$ given by
\[
\mathcal L_{k_n}^{(r_n-1)}\big(|z|^2\big)=(-1)^{k_n}\varphi_{k_n}(2z),\quad  z\in\R^{2r_n}.
\]
However, similar to the proof of \cite{MaMue14b}, the formulas \eqref{eq:conv-kernel} and \eqref{eq:conv-kernel-alt} can be easily transformed into each other when using the fact that the matrix coefficients of the Schrödinger representation (which are the Fourier--Wigner transform of a pair of Hermite functions) remain essentially invariant under the Fourier transform, see \cite[Equation (1.90)]{Fo89}.
\end{remark}

\section{Reduction of \texorpdfstring{\cref{thm:multiplier}}{Theorem 1.2} to dyadic spectral multipliers}\label{sec:abstract}

To prove \cref{thm:multiplier}, we use the following general spectral multiplier result of \cite{ChOuSiYa16}, which allows us to reduce the spectral multiplier estimates of \cref{thm:multiplier} to estimates for spectral multipliers whose Fourier transforms are supported on dyadic scales. Very similar arguments have already been used in \cite[Section\ 5]{ChOu16} and \cite[Section\ 4]{Ni22}. The proof of the result in \cite{ChOuSiYa16} is based on arguments from Calderón--Zygmund theory.

We choose a dyadic decomposition $(\chi_\iota)_{\iota\in\Z}$ of $\R\setminus\{0\}$ such that $\sum_{\iota\in\Z} \chi_\iota(\lambda) = 1$ for all $\lambda\neq 0$, and
\begin{equation}\label{eq:dyadic}
\chi_\iota(\lambda)=\chi(\lambda/2^\iota)
\end{equation}
for some even bump function $\chi:\R\to [0,1]$ supported in $[-2,-1/2]\cup [1/2,2]$. Given a suitable multiplier $F:\R\to\C$, we use the notation
\begin{equation}\label{eq:dyadic-piece}
F^{(\iota)} := (\hat F \chi_\iota)^\vee\quad \text{for }\iota\in\Z,
\end{equation}
where $\widehat\cdot$ and $\cdot^\vee$ denote the Fourier transform and its inverse on $\R$, respectively.

\begin{definition}
Let $(X,\rho,\mu)$ be a metric measure space with distance~$\rho$ and measure $\mu$, and $L$ be a positive self-adjoint operator on $L^2(X)$. We say that $L$ possesses the \textit{finite propagation speed property} if
\[
(\cos(t\sqrt L)f,g)_{L^2(X)} = 0\quad\text{for all } |t|< \rho(U,V)
\]
whenever $f,g\in L^2(X)$ are supported in open subsets $U,V\subseteq X$, where
\[
\rho(U,V) := \inf\{\rho(u,v):u\in U,v\in V \}
\]
denotes the distance of the two sets $U,V\subseteq X$.
\end{definition}

In the following, given some metric space $(X,\rho)$ we denote the ball of radius $r>0$ centered at $x\in X$ by $B_r(x)$.

\begin{proposition}\cite[Proposition I.22]{ChOuSiYa16}\label{prop:COSY-original}
Let $(X,\rho,\mu)$ be a metric measure space of homogeneous type and $Q\ge 0$ such that
\[
\mu(B_{\lambda r}(x)) \le C \lambda^Q \mu(B_{r}(x)) \quad \text{for all } x\in X, r>0,\lambda\ge 1.
\]
Let $1\le p_0 < p <2$. Suppose that $L$ is a positive self-adjoint operator on $L^2(X)$ such that the following statements are satisfied:
\begin{enumerate}
\item[(i)] $L$ possesses the finite propagation speed property.
\item[(ii)] $L$ satisfies the \textit{Stein--Tomas restriction type condition} $(\mathrm{ST}^\infty_{p_0, 2})$ of \cite{ChOuSiYa16}, that is, for any $R>0$ and all bounded Borel functions $F:\R\to\C$ supported in $[0,R]$,
\begin{equation}\label{eq:ST-condition}
\|F(\sqrt L)(\mathbf{1}_{B_r(x)}f) \|_2 \leq C \mu(B_r(x))^{{\frac 1 2}-{\frac 1 {p_0}}} ( Rr )^{Q({\frac 1 {p_0}}-{\frac 1 2})}\|F\|_\infty \|f\|_{p_0}
\end{equation}
for all $x\in X$, all $r\geq 1/R$, and all $f\in L^{p_0}(X)$.
\item[(iii)] There is some $\beta > Q/2$ such that
\[
\sup_{t>0}\|F(t\sqrt L)\|_{p\to p}\leq C\|F\|_{L^\infty_\beta}
\]
for all even bounded Borel functions $F:\R\to\C$ such that $\supp F \subseteq [-1, 1]$.
\end{enumerate}
Suppose that $F:\R\to\C$ is an even bounded Borel function and that there is a bounded sequence $(\alpha(\iota))_{\iota \in\Z}$ with $\sum_{\iota \ge 0}(\iota+1) \alpha(\iota)<\infty$ such that
\begin{equation}\label{eq:cond-1}
\norm{(F\chi_i)^{(j)}(\sqrt L)}_{p\to p} \le \alpha(i+j) \quad\text{for all } i,j\in \Z.
\end{equation}
Then the operator $F(\sqrt L)$ is of weak-type $(p,p)$.
\end{proposition}

\begin{remark}
Actually, in \cite{ChOuSiYa16}, Proposition~I.22 requires the condition $(\mathrm E_{p_0,2})$ in place of the Stein--Tomas type restriction condition $(\mathrm{ST}^\infty_{p_0, 2})$. However, both conditions are equivalent by Proposition~I.3 of the same paper. Moreover, note that the notation of our decomposition indexed by $i$ and $j$ differs slightly from that of \cite{ChOuSiYa16} since $(F\eta_i)^{(j)}=(F\chi_{-i})^{(j)}$, where $(\eta_i)_{i\in\Z}$ is the dyadic decomposition from \cite[Eq.\ (I.3.3)]{ChOuSiYa16}. The somewhat artificial requirement that $F$ shall be an even function will become apparent at the beginning of the proof of \cref{prop:reduced}.
\end{remark}

We apply \cref{prop:COSY-original} in the setting where $X=G$ is a two-step stratified Lie group and
\begin{equation}\label{def:sub-Laplacian-6}
L=-(X_1^2+\dots+X_{d_1}^2)
\end{equation}
is the sub-Laplacian associated with the basis $X_1,\dots,X_{d_1}$ of the first layer of the stratification $\g=\g_1\oplus \g_2$ of $G$. The measure $\mu$ in \cref{prop:COSY-original} will be the Lebesgue measure $|\cdot|$ on $G$, where we again identify $G\cong \g$ via the exponential map and $\g\cong \R^d$ by means of the bases $X_1,\dots,X_{d_1}$ and $U_1,\dots,U_{d_2}$. The distance $\rho$ is given by the Carnot--Carathéodory distance $d_{\mathrm{CC}}$ associated with the vector fields $X_1,\dots,X_{d_1}$. That is, for $g,h\in G$, the distance $d_{\mathrm{CC}}(g,h)$ is given by the infimum over all lengths of horizontal curves $\gamma:[0,1]\to G$ joining $g$ with $h$ (see for instance Section III.4 of \cite{VaSaCo92}). By the Chow--Rashevskii theorem \cite[Proposition~III.4.1]{VaSaCo92}, $d_{\mathrm{CC}}$ is indeed a metric on $M$, which induces the (Euclidean) topology of $G$.

It is well known that the Carnot--Carathéodory distance $d_{\mathrm{CC}}$ renders the measure space $(G,|\cdot|)$ a space of homogeneous type. This can be seen as follows: First, note that, since $X_1,\dots,X_{d_1}$ are left-invariant vector fields, $d_{\mathrm{CC}}$ is a left-invariant metric, meaning that
\begin{equation}\label{eq:cc-translation}
d_{\mathrm{CC}}(ag,ah) = d_{\mathrm{CC}}(g,h)\quad\text{for all } a,g,h\in G. 
\end{equation}
Moreover, by Lemma 1.4 of \cite{FoSt82},
\begin{equation}\label{eq:cc-equivalence}
d_{\mathrm{CC}}(g,h) \sim \norm{g^{-1}h} \quad \text{for all } g,h\in G, 
\end{equation}
where $\Vert\cdot\Vert$ is given by
\[
\norm{ (x,u) } = (|x|^4+|u|^2)^{1/4},\quad (x,u)\in G,
\]
which is a homogeneous norm in the sense of \cite[p.\ 8]{FoSt82} with respect to the family of dilations $(\delta_R)_{R>0}$ given by 
\begin{equation}\label{def:dilation}
\delta_R(x,u)=(Rx,R^2u).
\end{equation}
In combination with \eqref{eq:cc-translation}, the volume of a given ball $B_R^{d_{\mathrm{CC}}}(x,u)$ of radius $R>0$ centered at $(x,u)\in G$ can be estimated by
\begin{equation}\label{eq:cc-volume}
|B_R^{d_{\mathrm{CC}}}(x,u)| = R^{Q} |B_1^{d_{\mathrm{CC}}}(0)|.
\end{equation}
Hence, the metric measure space $(X,\rho,\mu)=(G,d_{\mathrm{CC}},|\cdot|)$ is indeed a space of homogeneous type, with homogeneous dimension $Q=d_1+2d_2$.

Moreover, the sub-Laplacian $L$ possesses the finite propagation speed property with respect to the Carnot--Carathéodory distance $d_{\mathrm{CC}}$. For a proof, see \cite{Me84}, or alternatively \cite[Corollary 6.3]{Mue04}. 

\begin{lemma}\label{lem:finite-prop}
Let $G$ be a two-step stratified Lie group and $L$ be a sub-Laplacian as in \cref{def:sub-Laplacian-6}.
Then $L$ possesses the finite propagation speed property.
\end{lemma}

As a consequence of \cref{prop:COSY-original}, we obtain:

\begin{corollary}\label{cor:reduction}
Let $G$ be a two-step stratified Lie group and $L$ be a sub-Laplacian as in \cref{def:sub-Laplacian-6}. Let $p_{*}\in[1,2]$ and $s>1/2$. Suppose that for all $1\le p\le p_{*}$ there exists some $\varepsilon>0$ such that
\begin{equation}\label{eq:reduced}
\Vert F^{(\iota)}(\sqrt L) \Vert_{p\to p} \le C_{p,s} 2^{-\varepsilon\iota} \norm{F}_{L^2_s}
\quad \text{for all } \iota \in\N
\end{equation}
and all even bounded Borel functions $F\in L^2_s(\R)$ supported in $[-2,-1/2]\cup[1/2,2]$. Then the statements (\ref{thm:part-1}) and (\ref{thm:part-2}) of \cref{thm:multiplier} hold for all $1\le p\le p_{*}$.
\end{corollary}

\begin{remark}
The assumption $1\le p\le p_{d_1,d_2}$ of \cref{thm:multiplier} automatically implies that $s>1/2$ if $s>d\left(1/p-1/2\right)$ since $p_{d_1,d_2}\le p_{d_2}=2(d_2+1)/(d_2+3)$ and thus
\[
d\bigg(\frac 1 p- \frac 1 2\bigg)\ge d_2\bigg(\frac{(d_2+3)}{2(d_2+1)}-\frac 1 2\bigg) = \frac{d_2}{d_2+1} \ge \frac 1 2.
\]
However, in \cref{cor:reduction}, we only require $1\le p\le p_{*}$ for some $p_{*}\in [1,2]$, which is why we additionally assume $s>1/2$ to make sure that $\|F|_{(0,\infty)}\|_\infty \lesssim \|F\|_{L^2_{s,\sloc}}$.
\end{remark}

\begin{proof}
See the proof of Corollary 6.2 in \cite{Ni24}.
\end{proof}

As a result, to prove \cref{thm:multiplier}, it suffices to show for $s>d\left(1/p-1/2\right)$ that
\[
\Vert F^{(\iota)}(\sqrt L) \Vert_{p\to p} \le C_{p,s} 2^{-\varepsilon\iota} \norm{F}_{L^2_s}
\quad \text{for all } \iota \in\N
\]
and all even bounded Borel functions $F\in L^2_s(\R)$ supported in $[-2,-1/2]\cup[1/2,2]$.

\section{On the numerology of Métivier groups} \label{sec:numerology}

From now on, we restrict to the case where $G$ is a Métivier group. Then, by definition, the skew-symmetric bilinear form $\omega_\mu(x,x') = \mu([x,x'])$ on $\g_1$ is non-degenerate for all $\mu\in\g_2^*\setminus\{0\}$. Let again $J_\mu$ denote the skew-symmetric endomorphism such that
\[
\omega_\mu(x,x') =\langle J_\mu x,x'\rangle\quad \text{for all } x,x' \in\g_1,
\]
where $\langle\cdot,\cdot\rangle$ denotes the inner product on $\g$ rendering $X_1,\dots,X_{d_1},U_1,\dots,U_{d_2}$ an orthonormal basis of $\g$. Then $G$ is a Métivier group if and only if $J_\mu$ is invertible for all $\mu\in \g_2^*\setminus\{0\}$. In particular, if $G$ is a Métivier group, then $d_1\in 2\N$, and $d_2 < d_1$ (since the map $\g_2^*\to(\g_1/\R x')^*,\mu\mapsto \omega_\mu(\cdot,x')$ is injective for $x'\neq 0$). But in fact, regarding the relationship of the dimensions $d_1$ and $d_2$, the dimension $d_1$ is in general much larger than $d_2$, as stated in \cref{prop:Radon-Hurwitz-skew} below.

Given $n\in\N\setminus\{0\}$, the \textit{Radon--Hurwitz number} $\rho_{\mathrm{RH}}(n)$ is defined as follows: If we write $n=\left(2a+1\right)2^b$ and $b=4q+r$, with $a,b,q,r\in\N$ and $0\le r<4$, then
\[
\rho_{\mathrm{RH}}(n) = 2^r + 8q.
\]

\begin{proposition}\label{prop:Radon-Hurwitz-skew}
Let $0<d_1,d_2<d$ be integers with $d_1+d_2=d$. Then there exists a Métivier group of dimension $d$ with center of dimension $d_2$ if and only if
\begin{equation}\label{eq:RH-dimensions}
d_2 < \rho_{\mathrm{RH}}(d_1).
\end{equation}
In particular, for any such Métivier group with $(d_1,d_2)\notin\{(4,3),(8,6),(8,7)\}$,
\begin{equation}\label{eq:dim-3/2}
d_1 > 3 d_2 /2.
\end{equation}
\end{proposition}

\begin{proof}
\cite[Corollary 1]{Ka80} asserts that condition \eqref{eq:RH-dimensions} is in fact equivalent to having even a Heisenberg type group with the requisite dimensions. So all we need to show is that \eqref{eq:RH-dimensions} holds also true for any Métivier group. Let $\varsigma(n)$ denote be the maximal number $l\in \N$ such that there are skew-symmetric matrices $J_1,\dots,J_l\in\R^{n\times n}$ satisfying
\[
J(\mu):=\sum_{j=1}^l \mu_j J_j \in \mathrm{GL}(n,\R) \quad \text{for all } \mu=(\mu_j)_j\neq 0.
\]
Note that any Métivier group with second layer of dimension $d_2$ gives rise to such a family of $d_2$ invertible skew-symmetric matrices $J_1,\dots,J_{d_2}\in\R^{d_1\times d_1}$ and vice versa. Thus, we only need to show $\varsigma(n) < \rho_{\mathrm{RH}}(n)$.

To that end, choose matrices $J_1,\dots,J_l\in\R^{n\times n}$ with $l=\varsigma(n)$ such that the linear combination $J(\mu)$ is invertible for all $\mu\neq 0$. Being skew-symmetric, the matrices $J_1,\dots,J_l$ give rise to a system of linearly independent vector fields $X_1,\dots,X_l$ on the Euclidean sphere $S^{n-1}\subseteq\R^{n}$ via $x\mapsto J_j x$. But according to \cite{Ad62}, there can only be at most $\rho_{\mathrm{RH}}(n)-1$ linearly independent vector fields on the sphere, whence we have $\varsigma(n) < \rho_{\mathrm{RH}}(n)$.

The inequality \eqref{eq:dim-3/2} is a direct consequence of \eqref{eq:RH-dimensions}. Suppose that $G$ is a Métivier group with first layer of dimension $d_1$ and second layer of dimension $d_2$. We write $d_1=\left(2a+1\right)2^b$ with $b=4q+r$, where $a,b,q,r\in\N$ and $0\le r<4$. Suppose first $a\ge 1$. In that case, since clearly $2^{4q+r+1}\ge 8q+2^r$, we have
\[
2d_1 = \left(2a+1\right)2^{b+1} \ge 3 \cdot 2^{4q+r+1} \ge 3 ( 8q+2^r ) = 3 \rho_{\mathrm{RH}}(d_1).
\]
Now suppose $a=0$. If $d_1\notin\{2,4,8\}$, then $q\ge 1$, and
\begin{align*}
2d_1 & = 2^{b+1} = 2^r (2^{4q+1}- 3) + 3 \cdot 2^r \\
& \ge 2^{4q+1}- 3 + 3 \cdot 2^r \ge 3 ( 8q+2^r  ) = 3 \rho_{\mathrm{RH}}(d_1) > 3d_2.
\end{align*}
A direct inspection of the remaining cases $d_1=2,4,8$ shows that \eqref{eq:dim-3/2} is fulfilled whenever $(d_1,d_2)\notin\{(4,3),(8,6),(8,7)\}$.
\end{proof}

\section{A first layer weighted Plancherel estimate} \label{sec:Plancherel}

Similar to \cite{Ni22,Ni24}, the proof of the spectral multiplier estimates of \cref{thm:multiplier} relies on weighted Plancherel estimates, which are exploited to turn support conditions of the involved convolution kernels into some sort of rapid decay. 

\begin{proposition}\label{prop:weighted-plancherel-1}
Suppose that $G$ is a Métivier group with layers of dimensions $d_1=\dim\g_1$ and $d_2=\dim \g_2$. If $F:\R\to\C$ is a bounded Borel function supported in a compact subset $A\subseteq (0,\infty)$ and $\chi:(0,\infty)\to\C$ is some bump function, the convolution kernel $\mathcal K_{F(L)\chi(2^\ell U)}$ of $F(L)\chi(2^\ell U)$ satisfies
\begin{equation}\label{eq:weighted-plancherel-1}
\int_G \big| |x|^\alpha \mathcal K_{F(L)\chi(2^\ell U)}(x,u) \big|^2 \,d(x,u) \le C_{A,\chi,\alpha} 2^{\ell(2\alpha-d_2)} \norm{F}_{L^2}^2
\end{equation}
for all $\alpha\ge 0$ and $\ell\in\Z$.
\end{proposition}

\begin{remark}\label{rem:1st-layer}
It should be pointed out that \cref{eq:weighted-plancherel-1} is a bound in terms of the $L^2$-norm of the multiplier $F$, which is crucial for the proof of \cref{prop:reduced} and one of the main reasons why passing to the setting of Métivier groups. Reproducing such an estimate beyond Métivier groups would require replacing $\norm{F}_{L^2}$ by some Sobolev norm of $F$, since multiplication by some weight $|x|^\alpha$ would then correspond to taking derivatives of the multiplier $F$, as the formula \cref{eq:conv-kernel} for the convolution kernel suggests. See also \cite[Theorem 2.7]{Ma12} for a similar estimate.
\end{remark}

\begin{proof}
Via interpolation, it suffices to prove \eqref{eq:weighted-plancherel-1} for $\alpha\in\N$. We use the notation of \cref{prop:rotation}. By \cref{prop:conv-kernel} and Plancherel's theorem, the left-hand side of \eqref{eq:weighted-plancherel-1} equals a constant times 
\[
\int_{\g_{2,r}^*} \int_{\g_1} |x|^{2\alpha} \bigg|\sum_{\mathbf{k}\in\N^N} F(\eigvp{k}{\mu})\chi(2^\ell|\mu|) \prod_{n=1}^N \varphi_{k_n}{(b_n^\mu,r_n)}(R_\mu^{-1} P_n^\mu x)\bigg|^2 \, dx\, d\mu,
\]
where $R_\mu\in O(d_1)$, and
\[
\eigvp{k}{\mu} = \sum_{n=1}^N \left(2 k_n +r_n\right)b_n^\mu,\quad \mathbf{k}=(k_1,\dots,k_N)\in \N^N.
\]
Recall that the orthogonal projections $P^\mu_n$ of \cref{prop:rotation} satisfy
\[
P_n^\mu R_\mu=R_\mu P_n\quad \text{for all }\mu \in \g_{2,r}^* \text{ and } n\in\{0,\dots,N\},
\]
where $P_n$ denotes the projection from $\R^{d_1}=\R^{r_0}\oplus\R^{2r_1}\oplus \dots\oplus \R^{2r_N}$ onto the $n$-th layer. Subsitutiting $y=R_\mu^{-1}x$, we see that the above expression equals
\[
\int_{\g_{2,r}^*} \int_{\R^{d_1}} |y|^{2\alpha} \bigg|\sum_{\mathbf{k}\in\N^N} F(\eigvp{k}{\mu})\chi(2^\ell|\mu|) \prod_{n=1}^N \varphi_{k_n}^{(b_n^\mu,r_n)}(P_n y)\bigg|^2 \, dy\, d\mu.
\]
Hence, expanding
\[
|y|^{2\alpha}=(|P_1 y|^2+\dots+|P_N y|^2)^{\alpha}
\]
shows that it is sufficient to show that every term of the form
\begin{equation}\label{eq:weight-expansion-1}
\int_{\g_{2,r}^*} \int_{\R^{d_1}} \bigg|\bigg(\prod_{n=1}^N |P_n y|^{m_n}\bigg)\sum_{\mathbf{k}\in\N^N} F(\eigvp{k}{\mu})\chi(2^\ell|\mu|) \prod_{n=1}^N \varphi_{k_n}^{(b_n^\mu,r_n)}(P_n y)\bigg|^2 \, dx \, d\mu
\end{equation}
with $m_1,\dots,m_N\in\N$ satisfying $m_1+\dots+m_N=\alpha$ is bounded by the right-hand side of \eqref{eq:weighted-plancherel-1}.
For $m\in\N\setminus\{0\}$ and $\lambda>0$, let
\[
H^{(\lambda)}_{\R^{2m}}=- \Delta_z + \tfrac{\lambda^2} 4 |z|^2
\]
denote the (rescaled) Hermite operator on $\R^{2m}$. By \cite[Proposition 3.3]{ChOu16} (or alternatively \cite{He93}), the Hermite operator satisfies the sub-elliptic estimate
\begin{equation}\label{eq:sub-elliptic}
\norm*{|\cdot|^\beta f}_{L^2(\R^{2m})} \lesssim \lambda^{-\beta} \big\| \big(H_{\R^{2m}}^{(\lambda)}\big)^{\beta/2} f\big\|_{L^2(\R^{2m})}, \quad\beta\ge 0.
\end{equation}
By Equations (1.3.25) and (1.3.42) of \cite{Th93}, the Laguerre functions $\smash{\varphi_k^{(\lambda,m)}}$ are eigenfunctions of the $\lambda$-rescaled Hermite operator, that is,
\[
H^{(\lambda)}_{\R^{2m}} \varphi_k^{(\lambda,m)} = \left(2k+m\right) \lambda \, \varphi_k^{(\lambda,m)},\quad k\in \N.
\]
Using the sub-elliptic estimate \eqref{eq:sub-elliptic} on every block of $\R^{d_1}=\R^{2r_1}\oplus\dots\oplus\R^{2r_N}$, the expanded term \cref{eq:weight-expansion-1} can thus be dominated by a constant times
\begin{align}
\int_{\g_{2,r}^*} \int_{\R^{d_1}} & \bigg|\sum_{\mathbf{k}\in\N^N} F(\eigvp{k}{\mu})\chi(2^\ell|\mu|)\notag \\
& \times\prod_{n=1}^N \left(b_n^\mu\right)^{-m_n}\left(\left(2k_n+r_n\right)b_n^\mu\right)^{m_n/2} \varphi_{k_n}^{(b_n^\mu,r_n)}(P_n y)\bigg|^2 \, dy \, d\mu.\label{eq:weight-expansion-2}
\end{align}
Since $G$ is a Métivier group, the endomorphism $J_\mu$ given by
\[
\langle J_\mu x,x'\rangle =\mu([x,x']),\quad x,x'\in \g_1
\]
is invertible for all $\mu\in\g_2^*\setminus\{0\}$. By \cref{prop:rotation}, since $b_1^\mu,\dots,b_N^\mu$ are the non-negative eigenvalues of $iJ_\mu$, this necessarily means that $b_n^\mu\neq 0$ for all $\mu\in\g_2^*\setminus\{0\}$. Thus, the image $\textbf{B}$ of the unit sphere $S^{d_2-1}$ under the continuous map $\mu\mapsto \textbf{b}^\mu=(b_1^\mu,\dots,b_N^\mu)$ is a compact subset of $(0,\infty)^N$. In particular,
\begin{equation}\label{eq:b_n-metivier}
b_n^\mu \sim_{\textbf{B},\textbf{r}} 1 \quad \text{for all } \mu\in S^{d_2-1} \text{ and } n\in\{1,\dots,N\}.
\end{equation}
Thus, if $\eigvp{k}{\mu}\in\supp F$ and $2^\ell|\mu|\in\supp \chi$, then
\[
\left(2k_n + r_n\right) b_n^\mu
  \le \eigvp{k}{\mu} \lesssim_A 1
\]
and, since the functions $\mu\mapsto b_n^\mu$ are homogeneous of degree 1 by \cref{prop:rotation},
\[
b_n^\mu = \left|\mu\right| b_n^{\bar \mu} \sim_{\textbf{B},\textbf{r}} |\mu| \sim 2^{-\ell},\quad \text{where } \bar \mu := \mu/|\mu|.
\]
Hence, using orthogonality on $\R^{d_1}$, \cref{eq:weight-expansion-2} can be estimated by a constant times
\begin{equation}\label{eq:weight-expansion-3}
(2^\ell)^{2\alpha} \sum_{\mathbf{k}\in\N^N} \int_{\g_{2,r}^*} |F(\eigvp{k}{\mu})\chi(2^\ell|\mu|)|^2 \prod_{n=1}^N \big\| \varphi_{k_n}^{(b_n^\mu,r_n)}\big\|_{L^2(\R^{2r_n})}^2 \, d\mu.
\end{equation}
By Theorem 1.3.2 and Equation (1.3.42) of \cite{Th93}, we have
\begin{align*}
\big\|\varphi_{k_n}^{(b_n^\mu,r_n)}\big\|_{L^2(\R^{2r_n})}^2
& \sim \left(b_n^\mu\right)^{r_n} \binom{k_n+r_n-1}{k_n} \\
& \sim \left(b_n^\mu\right)^{r_n} \left(k_n+1\right)^{r_n-1}
\lesssim b_n^\mu.
\end{align*}
Thus, \eqref{eq:weight-expansion-3} can be bounded by a constant times
\begin{equation}\label{eq:plancherel-3}
(2^\ell)^{2\alpha} \sum_{\mathbf{k}\in\N^N} \int_{\g_{2,r}^*} |F(\eigvp{k}{\mu})\chi(2^\ell|\mu|)|^2 \prod_{n=1}^N b_n^\mu \, d\mu.
\end{equation}
We rewrite the integral in \cref{eq:plancherel-3} in polar coordinates, that is,
\[
\mu = \rho \omega \quad\text{for } \rho \in [0,\infty) \text{ and } |\omega|=1.
\]
Then, since $\mu\mapsto\textbf{b}^\mu$ is homogeneous of degree 1, we have $\eigvp{k}{\mu} = \rho \eigvp{k}{\omega}$. Thus, using $\rho=|\mu|\sim 2^{-\ell}$, \cref{eq:plancherel-3} is bounded by a constant times
\[
2^{-\ell (2\alpha+d_2+N)} \int_{S^{d_2-1}}  \int_0^\infty \sum_{\textbf{k} \in \mathbb{N}^N}  \left|F(\rho\eigvp{k}{\omega})\chi(2^\ell\rho)\right|^2 
 \prod_{n=1}^N b_n^\omega \,\frac{d\rho}{\rho} \,d\sigma(\omega) .
\]
Substituting $\rho=(\eigvp{k}{\omega})^{-1}\lambda$ in the inner integral, we see that the above term equals
\[
2^{-\ell (2\alpha+d_2+N)} \int_{S^{d_2-1}} \int_0^\infty \sum_{\textbf{k} \in \mathbb{N}^N}\left|F(\lambda)\chi(2^\ell (\eigvp{k}{\omega})^{-1} \lambda)\right|^2
 \prod_{n=1}^Nb_n^{\omega} \,\frac{d \lambda}\lambda \,d\sigma(\omega) 
\]
If $2^\ell (\eigvp{k}{\omega})^{-1} \lambda\in\supp\chi$, then $(2k_n+r_n) b_n^\omega \le \eigvp{k}{\omega} \sim 2^\ell \lambda$. Thus, $k_n\lesssim 2^\ell \lambda (b_n^\omega)^{-1}$ for the non-vanishing summands in the above sum over $\textbf{k}$. Hence, the above term is bounded by a constant times
\[
2^{-\ell ( 2\alpha + d_2)} \int_0^\infty \left|F(\lambda)\right|^2 \lambda^N \frac{d \lambda}\lambda,
\]
which is comparable to $2^{-\ell d_2} \|F\|^2_2$ since $F$ is compactly supported.
\end{proof}

\section{\texorpdfstring{Proof of the reduction of \cref{thm:multiplier}}{Proof of the reduction of Theorem 1.2}} \label{sec:1st-layer}

In this section we prove \cref{thm:multiplier} using the reduction of \cref{sec:abstract}. However, in the case $(d_1,d_2)=(4,3)$, the arguments of this section only yield results in the range $1\le p\le 6/5$. We will improve this range to $1\le p\le 4/3$ in the next section.

Recall that \cref{prop:Radon-Hurwitz-skew} asserts that the dimensions $d_1=\dim\g_1$ and $d_2=\dim\g_2$ of the stratification $\g=\g_1\oplus\g_2$ of the Lie algebra of $G$ satisfy $d_1>3d_2/2$ whenever $G$ is a Métivier group, except for the cases where
\[
(d_1,d_2)\in\{(4,3),(8,6),(8,7)\}.
\]
Apart from these three exceptional cases, the arguments of this section yield spectral multiplier estimates for the range $1\le p\le p_{d_2}=2(d_2+1)/(d_2+3)$. In the exceptional cases, we also get results, but only for smaller ranges of $p$. To take this numerology into account, we let $\bar p_{d_1,d_2}=p_{d_2}$ for $(d_1,d_2)\notin\{(4,3),(8,6),(8,7)\}$, $\bar p_{4,3}= 6/5$, $\bar p_{8,6}=17/12$, and $\bar p_{8,7}=14/11$. Then, in view of \cref{cor:reduction}, to prove \cref{thm:multiplier} for $(d_1,d_2)\neq (4,3)$, it suffices to show the following:

\begin{proposition}\label{prop:reduced}
Let $G$ be a Métivier group with layers of dimension $d_1$ and $d_2$. Suppose that $1\le p\le \bar p_{d_1,d_2}$. If $s>d\left( 1/p - 1/2\right)$, then there exists some $\varepsilon>0$ such that
\[
\Vert F^{(\iota)}(\sqrt L) \Vert_{p\to p} \le C_{p,s} 2^{-\varepsilon\iota} \norm{F^{(\iota)}}_{L^2_s}\quad \text{for all } \iota \in\N
\]
and any even bounded Borel function $F\in L^2_s$ supported in $[-2,-1/2]\cup [1/2,2]$.
\end{proposition}

Here we use again the notation $F^{(\iota)} = (\hat F \chi_\iota)^\vee=F*\chi_\iota^\vee$ of \cref{eq:dyadic-piece}. A key tool for proving \cref{prop:reduced} is the finite propagation speed discussed in \cref{sec:abstract}. Let $d_{\mathrm{CC}}$ denote again the Carnot--Carathéodory distance associated with the vector fields $X_1,\dots,X_{d_1}$ giving rise to the sub-Laplacian $L=-(X_1^2+\dots+X_{d_1}^2)$. If $F$ is an even function, then so is its Fourier transform $\smash{\hat F}$. Since $\chi_\iota$ is also even, the Fourier inversion formula yields
\[
F^{(\iota)} ( \sqrt L) f = \frac{1}{2\pi} \int_{2^{\iota-1}\le |\tau|\le 2^{\iota+1}} \chi_\iota(\tau) \hat F(\tau) \cos(\tau \sqrt L)f \,d\tau.
\]
Suppose that $f$ is supported in a Carnot--Carathéodory ball $B_{R}^{d_{\mathrm{CC}}}(x,u)$ with radius $R=2^\iota$. Then, since $L$ has the finite propagation speed property with respect to $d_{\mathrm{CC}}$, the above identify shows that $F^{(\iota)} ( \sqrt L)f$ is supported in the ball $B_{3R}^{d_{\mathrm{CC}}}(x,u)$, which means that the convolution kernel $\mathcal K^{(\iota)}$ associated with the operator $F^{(\iota)}(L)$ must be supported in the Carnot--Carathéodory ball $B_{2R}^{d_{\mathrm{CC}}}(0,0)$ centered at the origin. On the other hand, by \cref{eq:cc-equivalence}, there exists a constant $C>0$ such that
\begin{equation}\label{eq:C}
B_R^{d_{\mathrm{CC}}}(0) \subseteq B_{CR}(0) \times B_{CR^2}(0)\subseteq\R^{d_1}\times \R^{d_2}.
\end{equation}
In particular, this means that the convolution kernel $\mathcal K^{(\iota)}$ is supported in a Euclidean ball of size $R\times R^2$, up to constants.

Let $U_1,\dots,U_{d_2}$ be a basis of $\g_2$, and $U=|\textbf{U}|^{1/2}$, where $\textbf{U}$ is the vector of differential operators $\textbf{U}=(-iU_1,\dots,-iU_{d_2})$. Let $\chi$ be the bump function of \cref{eq:dyadic} that gives rise to the dyadic decomposition $(\chi_\iota)_{\iota\in\Z}$ of $\R\setminus\{0\}$.
When passing to the operators \[F^{(\iota)}_\ell(L,U)=(F^{(\iota)}\psi)(\sqrt L)\chi(2^\ell U),\] where $\psi$ is some bump function supported away from the origin, the following phenomenon occurs. By \cref{prop:conv-kernel}, we know that the convolution kernel $\mathcal K^{(\iota)}_\ell$ of $F^{(\iota)}_\ell(L,U)$ can be written as
\begin{align*}
\mathcal K^{(\iota)}_\ell(x,u) = (2\pi)^{-d_2} \int_{\g_{2,r}^*} & \sum_{\mathbf{k}\in\N^N} (F^{(\iota)}\psi)\big(\sqrt{\eigvp{k}{\mu}}\big) \chi(2^\ell |\mu|) \\
& \times \bigg[\prod_{n=1}^N \varphi_{k_n}^{(b_n^\mu,r_n)}(R_\mu^{-1} P_n^\mu x)\bigg] e^{i\langle \mu, u\rangle} \, d\mu,
\end{align*}
where
\[
\eigvp{k}{\mu} = \sum_{n=1}^N \left(2 k_n +r_n\right)b_n^\mu,\quad \mathbf{k}=(k_1,\dots,k_N)\in \N^N.
\]
Note that $\eigvp{k}{\mu}\sim 1$ and $|\mu|\sim 2^{-\ell}$ on the support of the involved multipliers. As observed in \cref{eq:b_n-metivier}, since $G$ is a Métivier group, we have $b_n^\mu \sim 1$ for $|\mu|=1$. Since the functions $\mu\mapsto b_n^\mu$ are homogeneous of degree $1$, the sum above ranges in fact over all $\mathbf{k}\in\N^N$ satisfying
\[
|\textbf{k}|+1 \sim \eigvp{k}{\bar\mu} = |\mu|^{-1}\eigvp{k}{\mu} \sim |\mu|^{-1}\sim 2^\ell,\quad\text{where } \bar \mu = \frac{\mu}{|\mu|}.
\]
Thus, the estimates of \cite[Lemma~1.5.3]{Th93} show that the rescaled Laguerre functions $\smash{\varphi_{k_n}^{(b_n^\mu,r_n)}}$ are essentially supported in a ball of radius comparable to
\[
|\mu|^{-1/2}(k_n+1)^{1/2}\lesssim 2^\ell
\]
centered at the origin, up to some exponentially decaying term. Hence, we can think of the convolution kernel $\mathcal K^{(\iota)}_\ell$ as being essentially supported in a ball of dimensions comparable to $2^\ell \times R^2$ instead of $R\times R^2$, which is a smaller set if $\ell\le \iota$.

In the next lemma, we make the above argument more precise and show that we can essentially assume that the above convolution kernel is supported in a ball of dimensions comparable to $R_\ell R^\gamma \times R^2$, where $R_\ell=2^\ell$. The parameter $\gamma>0$ is crucial to prove the rapid decay, and will be chosen sufficiently small later in the proof of \cref{prop:reduced}. The constant $C>0$ in the following statement is given by \cref{eq:C} above.

\begin{lemma}\label{lem:rapid-decay}
Suppose that $1\le p< 2$. Consider the radii $R=2^\iota$ and $R_\ell=2^\ell$ for $\iota\in\N$ and $\ell \in\{-\ell_0,\dots,\iota\}$. Suppose that $F:\R\to\C$ is a bounded Borel function supported in a compact subset $A\subseteq (0,\infty)$ and $\chi:(0,\infty)\to\C$ is some bump function. Then, for every $\gamma>0$, if $f\in L^p(G)\cap L^2(G)$ is a function supported in the ball
\[
B_0^{(\ell)} = B_{CR_\ell}(x_{0}^{(\ell)}) \times B_{CR^2}(0) \subseteq \g_1 \times \g_2 ,
\]
then the $L^p$-norm of $\mathbf{1}_{B_{3R}^{d_{\mathrm{CC}}}(0)} F(L) \chi(2^\ell U) f$ is rapidly decaying in $R$ outside the ball 
\[
\tilde B_{0}^{(\ell)} = B_{2 C R_\ell R^\gamma}(x_{0}^{(\ell)}) \times B_{9 C R^2}(0)  \subseteq \g_1 \times \g_2.
\]
That is, for any $N\in\N$, there is a constant $C_{N,\gamma}>0$ such that
\begin{equation}\label{eq:rapid-decay}
\big\|\mathbf{1}_{(\g\setminus\tilde B^{(\ell)}_0)\cap B_{3R}^{d_{\mathrm{CC}}}(0)} F(L) \chi(2^\ell U) f \big\|_p \le C_{N,\gamma} R^{-N} \|F\|_\infty \|f\|_p.
\end{equation}
\end{lemma}

\begin{remark}\label{rem:sobolev}
An inspection of the following proof shows that the norm $\|F\|_\infty$ in \cref{eq:rapid-decay} could actually be replaced by the norm $\|F\|_q$ with $1/q=1/p-1/2$. However, this is not significant for our purposes, since, up multiplication by some bump function $\psi$, we will use \cref{eq:rapid-decay} for the multiplier $F^{(\iota)}$, in which case we have  $\|F^{(\iota)}\|_\infty \lesssim R^\alpha \norm{F^{(\iota)}}_2$ for $\alpha>1/2$ by the Sobolev embedding theorem.
\end{remark}

\begin{proof}
We interpolate between an $L^1$ and an $L^2$-estimate using the Riesz--Thorin interpolation theorem. The rapid decay comes from the $L^1$-estimate, while for the $L^2$-estimate, we neglect all cut-off functions and just use the trivial estimate
\begin{equation}\label{eq:error-L2}
\big\|\mathbf{1}_{(\g\setminus\tilde B^{(\ell)}_0)\cap B_{3R}^{d_{\mathrm{CC}}}(0)} F(L) \chi(2^\ell U) f \big\|_2
 \lesssim_\chi \|F\|_\infty \|f\|_2.
\end{equation}
To prove the $L^1$-estimate, we consider the sets
\[
A_\ell(x') := \{ (x,u)\in B_{3R}^{d_{\mathrm{CC}}}(0) : |x-x'|\ge C R_\ell R^\gamma \},\quad x'\in\g_1.
\]
Note that a given point $(x,u)\in G$ lies in $A_\ell(x')$ whenever
\[
(x,u)\in (\g\setminus\tilde B^{(\ell)}_0)\cap B_{3R}^{d_{\mathrm{CC}}}(0) \quad \text{and}\quad (x',u')\in \supp f.
\]
By \cref{prop:conv-kernel}, we can write
\begin{align*}
& F(L) \chi(2^\ell U) f(x,u)
 = f * \mathcal K_\ell(x,u) \\
& = \int_G f(x',u') \mathcal K_\ell\big((x',u')^{-1}(x,u)\big)\,d(x',u') ,\quad (x,u)\in G,
\end{align*}
where $\mathcal K_\ell$ denotes the convolution kernel of $F(L) \chi(2^\ell U)$. Let
\[
\kappa_\gamma(x',u') := \int_{A_\ell(x')} \big|\mathcal K_\ell\big((x',u')^{-1}(x,u)\big) \big| \,d(x,u).
\]
Then, using Fubini's theorem, we get
\begin{align}
& \big\|\mathbf{1}_{(\g\setminus\tilde B^{(\ell)}_0)\cap B_{3R}^{d_{\mathrm{CC}}}(0)} F(L) \chi(2^\ell U) f \big\|_1 \notag \\
& \le \int_G \mathbf{1}_{(\g\setminus\tilde B^{(\ell)}_0)\cap B_{3R}^{d_{\mathrm{CC}}}(0)}(x,u) \,\big| f * \mathcal  K_\ell(x,u)\big| \,d(x,u) \notag \\
& \le \int_{\supp f} \int_{A_\ell(x')} \big|f(x',u')\big| \, \big|\mathcal K_\ell\big((x',u')^{-1}(x,u)\big)\big| \,d(x,u)\,d(x',u') \notag \\
& = \int_{\supp f} \big|f(x',u')\big|\, \kappa_\gamma(x',u') \,d(x',u'). \label{eq:error-L1-2}
\end{align}
Recall that $|B_{3R}^{d_{\mathrm{CC}}}(0)|\sim R^Q$ by \cref{eq:cc-volume}, where $Q=d_1+2d_2$ denotes the homogeneous dimension of $G$. Given $N\in\N$, using that $|x-x'|\ge CR_\ell R^\gamma$ for $(x,u)\in A_\ell(x')$ and the Cauchy--Schwarz inequality, we get
\begin{align}
\kappa_\gamma(x',u') 
& = \int_{A_\ell(x')} \big|\mathcal K_\ell\big((x',u')^{-1}(x,u)\big)\big| \,d(x,u) \notag \\
& \lesssim (R_\ell R^\gamma)^{-N} \int_{A_\ell(x')} \big||x-x'|^N \mathcal K_\ell\big(x-x',u-u'-\tfrac 1 2 [x',x] \big) \big| \,d(x,u) \notag \\
& \lesssim (R_\ell R^\gamma)^{-N} R^{Q/2} \bigg( \int_G \big||x|^N \mathcal K_\ell(x,u) \big|^2 \,d(x,u) \bigg)^{1/2}. \label{eq:error-CS}
\end{align}
By \cref{prop:weighted-plancherel-1}, the second factor of \cref{eq:error-CS} can be estimated by
\[
\int_G \big||x|^N \mathcal K_\ell(x,u) \big|^2 \,d(x,u)
 \lesssim_N R_\ell^{2N-d_2} \|F\|_2^2.
\]
Plugging this estimate into \cref{eq:error-CS} yields
\[
\kappa_\gamma(x',u')
\lesssim_N R^{-\gamma N + Q/2} R_\ell^{-d_2/2} \|F\|_2.
\]
Altogether, with \cref{eq:error-L1-2}, we end up with
\begin{equation}\label{eq:error-L1}
\big\|\mathbf{1}_{(\g\setminus\tilde B^{(\ell)}_0)\cap B_{3R}^{d_{\mathrm{CC}}}(0)} F(L) \chi(2^\ell U) f \big\|_1
\lesssim_N R^{- \gamma N + Q/2} R_\ell^{-d_2/2} \norm{F}_2 \norm{f}_1.
\end{equation}
Using the Riesz--Thorin interpolation theorem together with \cref{eq:error-L1} and \cref{eq:error-L2} and choosing $N=N(\gamma)\in \N$ in \cref{eq:error-L1} sufficiently large yields \cref{eq:rapid-decay}.
\end{proof}

We now prove the spectral multiplier estimate of \cref{prop:reduced} and thus \cref{thm:multiplier}, except for the case $(d_1,d_2)=(4,3)$, which will be treated in the next section. If $(d_1,d_2)\notin\{(4,3),(8,6),(8,7)\}$, it suffices to use the weaker estimate
\begin{equation}\label{eq:restriction-type-i}
\norm*{ F(L)\chi(2^\ell U) }_{p\to 2}
\le C_{p,\chi} 2^{-\ell d_2(\frac 1p - \frac 1 2)} \norm*{F}_{2^{\ell},2}, \quad\ell\in\Z,
\end{equation}
which follows directly from \cref{thm:restriction-type} by using $\|F\|_2 \le \norm*{F}_{2^{\ell},2}$. The full strength of \cref{thm:restriction-type} is only needed for the exceptional cases.

\begin{proof}[Proof of \cref{prop:reduced}]
Let $\iota\in\N$ and $R=2^\iota$. We first argue as in Step (1) and (2) of the proof of Proposition 7.1 in \cite{Ni24}.

\smallskip

(1) \textit{Reduction to compactly supported functions.} Let $f\in D(G)$ be an integrable simple function on $G$. We first show that we may restrict to the case where $f$ is supported in $B_R^{d_{\mathrm{CC}}}(0)$. Since the metric space $(G,{d_{\mathrm{CC}}})$ endowed with the Lebesgue measure is a space of homogeneous type and separable, we may choose a decomposition into disjoint sets $B_j \subseteq B_R^{d_{\mathrm{CC}}}(x^{(j)},u^{(j)})$, $j\in\N$, $(x^{(j)},u^{(j)})\in G$ such that for every $\lambda\ge 1$, the number of overlapping dilated balls $B_{\lambda R}^{d_{\mathrm{CC}}}(x^{(j)},u^{(j)})$ is bounded by a constant $C(\lambda)\sim \lambda^Q$, which is independent of $\iota$. We decompose $f$ as
\[
f = \sum_{j=0}^\infty f_j\quad \text{where } f_j:=f|_{B_j}.
\]
Since $F$ is even, so is $\hat F$. As $\chi_\iota$ is even as well, the Fourier inversion formula provides
\[
F^{(\iota)} ( \sqrt L)f_j = \frac{1}{2\pi} \int_{2^{\iota-1}\le |\tau|\le 2^{\iota+1}} \chi_\iota(\tau) \hat F(\tau) \cos(\tau \sqrt L)f_j \,d\tau.
\]
Since $L$ satisfies the finite propagation speed property, $F^{(\iota)} ( \sqrt L)f_j$ is supported in $B_{3R}^{d_{\mathrm{CC}}}(x^{(j)},u^{(j)})$ by the formula above. Together with the bounded overlap of these balls, we obtain
\[
\norm{F^{(\iota)} ( \sqrt L)f}_p^p \lesssim \sum_{j=0}^\infty \Vert F^{(\iota)} ( \sqrt L)f_j \Vert_p^p.
\]
Altogether, since $L$ is left-invariant, it suffices to show
\begin{equation}\label{eq:reduction-1}
\big\| \textbf{1}_{B_{3R}^{d_{\mathrm{CC}}}(0)} F^{(\iota)}( \sqrt L) f \big\|_p \lesssim 2^{-\varepsilon\iota} \norm{F^{(\iota)}}_{L^2_s} \norm{f}_p
\end{equation}
whenever our initial function $f\in D(G)$ is supported in $B_R^{d_{\mathrm{CC}}}(0)$.

\smallskip

(2) \textit{Localizing the multiplier.} Next we show that we may replace the multiplier $F^{(\iota)}$ by $F^{(\iota)}\psi$, where $\psi$ is a smooth cut-off function which is compactly supported away from the origin. Using the dyadic decomposition $(\chi_\iota)_{\iota \in\Z}$ of \cref{eq:dyadic}, we put
\[
\psi:=\sum_{j=-2}^2 \chi_j.
\]
Then $1/8\le |\lambda| \le 8$ whenever $\lambda\in\supp\psi$, and $|\lambda| \notin (1/4,4)$ if $\lambda\in\supp(1-\psi)$. We decompose $F^{(\iota)}$ as
\[
F^{(\iota)}=F^{(\iota)}\psi + F^{(\iota)}(1-\psi).
\]
The second part of this decomposition can be treated by the Mikhlin--Hörmander type result of \cite{Ch91} and \cite{MaMe90}. Note that $F^{(\iota)}=F*\check \chi_\iota$, and $\check\chi\in\S(\R)$. Thus, given $\alpha\in\N$ and $N\in\N$, we have
\begin{align}
\Big|\Big(\frac{d}{d\lambda}\Big)^\alpha  F^{(\iota)} (\lambda )\Big|
 & = \Big|\Big(\frac{d}{d\lambda}\Big)^\alpha \int_{-2}^2 2^\iota F(\tau) \check \chi(2^\iota(\lambda-\tau))\, d\tau  \Big| \notag \\
 & \lesssim_N 2^{\iota(\alpha+1)} \int_{-2}^2 \frac{|F(\tau)|}{(1+2^\iota|\lambda-\tau|)^N}\, d\tau. \label{eq:error-conv-i}
\end{align}
Since $F$ is supported in $[-2,-1/2]\cup [1/2,2]$, choosing $N:=\alpha+2$ in \cref{eq:error-conv-i} gives
\[
\Big|\Big(\frac{d}{d\lambda}\Big)^\alpha  F^{(\iota)} (\lambda )\Big| \lesssim  2^{-\iota} \min\{|\lambda|^{-\alpha},1\} \norm{F}_2 \quad \text{whenever } |\lambda|\notin (1/4,4).
\]
This implies
\[
\norm{F^{(\iota)}(1-\psi)}_{L^2_{Q/2+1,\sloc}} \lesssim_{\psi} 2^{-\iota} \norm{F}_2.
\]
Hence, the Mikhlin--Hörmander type result of \cite{Ch91} and \cite{MaMe90} yields
\[
\Vert (F^{(\iota)}(1-\psi))(\sqrt L)\Vert_{p\to p}
 \lesssim 2^{-\iota} \norm{F}_2.
\]
Thus, instead of \cref{eq:reduction-1}, we are left proving
\begin{equation}
\big\| \textbf{1}_{B_{3R}^{d_{\mathrm{CC}}}(0)} (F^{(\iota)}\psi)(\sqrt L) f \big\|_p \lesssim 2^{-\varepsilon\iota} \norm{F^{(\iota)}}_{L^2_s} \norm{f}_p \label{eq:reduction-2}
\end{equation}
for all $f\in D(G)$ that are supported in $B_R^{d_{\mathrm{CC}}}(0)$.

\smallskip

(3) \textit{Truncation along the spectrum of $U$.} We decompose $(F^{(\iota)}\psi)(\sqrt L)$ by a dyadic decomposition of $U$. For $\ell\in\Z$, let $F_\ell^{(\iota)} : \R\times \R \to \C$ be given by
\[
F_\ell^{(\iota)}(\lambda,\rho) = (F^{(\iota)}\psi) (\sqrt \lambda)\chi(2^\ell\rho)\quad\text{for }\lambda\geq 0
\]
and $F_\ell^{(\iota)}(\lambda,\rho)=0$ else. Moreover, as observed in \cref{rem:ell_0}, there is a natural number $\ell_0\in\N$ depending only on the matrices $J_\mu$ and the inner product $\langle\cdot,\cdot\rangle$ on $\g$ such that
\[
\sum_{\ell \ge - \ell_0} F_\ell^{(\iota)}(L,U) f = F^{(\iota)}(\sqrt L)f.
\]
We decompose the function on the left-hand side of \eqref{eq:reduction-2} as
\begin{align*}
\mathbf{1}_{B_{3R}^{d_{\mathrm{CC}}}(0)} (F^{(\iota)}\psi)(\sqrt L) f
& = \mathbf{1}_{B_{3R}^{d_{\mathrm{CC}}}(0)} \bigg(\sum_{\ell =-\ell_0}^\iota + \sum_{\ell = \iota + 1}^\infty \bigg) F_\ell^{(\iota)}(L,U) f.
\end{align*}

(4) \textit{The case $\ell>\iota$.} The second part of the decomposition can be directly treated by the restriction type estimate \cref{eq:restriction-type-i}, which yields
\begin{equation}\label{eq:restriction-type-ia}
\| F_\ell^{(\iota)}(L,U) \|_{p\to 2}
\lesssim 2^{-\ell d_2(\frac 1p - \frac 1 2)} \| (F^{(\iota)}\psi)(\sqrt{\cdot}\,)\|_{2^{\ell},2}.
\end{equation}
By \cref{eq:norm}, using that $\psi$ is compactly supported away from the origin, we get
\begin{equation}\label{eq:sobolev-type-ii}
\| (F^{(\iota)}\psi)(\sqrt{\cdot}\,)\|_{2^{\ell},2} \lesssim_{\tilde s} \| F^{(\iota)} \|_2 + 2^{-\ell \tilde s} \| F^{(\iota)} \|_{L^2_{\tilde s}} \quad \text{for every } \tilde s>1/2.
\end{equation}
Recall that $|B_R^{d_{\mathrm{CC}}}(0)| \sim R^Q$ by \eqref{eq:cc-volume}. Hence, using Hölder's inequality with $1/q=1/p-1/2$, as well as \cref{eq:restriction-type-ia} and \cref{eq:sobolev-type-ii}, we get
\begin{align*}
& \bigg \Vert  \mathbf{1}_{B_{3R}^{d_{\mathrm{CC}}}(0)} \sum_{\ell = \iota + 1}^\infty F_\ell^{(\iota)}(L,U) f \bigg \Vert_p
   \lesssim R^{Q/q} \sum_{\ell = \iota + 1}^\infty \Vert F_\ell^{(\iota)}(L,U) f \Vert_2 \\
 & \lesssim 2^{\iota Q/q}\sum_{\ell = \iota + 1}^\infty 2^{-\ell d_2/q} \big ( \| F^{(\iota)} \|_2 + 2^{-\ell \tilde s} \| F^{(\iota)} \|_{L^2_{\tilde s}} \big) \norm{f}_p.
\end{align*}
Summation of the geometric series shows that the last term is bounded by a constant depending on $\psi$ times
\[
2^{\iota d/q}  \big( \norm{F^{(\iota)}}_2 + 2^{-\iota \tilde s} \norm{F^{(\iota)}}_{L^2_{\tilde s}} \big) \norm{f}_p,
\]
which in turn is bounded by $2^{-\varepsilon \iota} \|F^{(\iota)}\|_{L^2_{\tilde s}}\norm{f}_p$ if we choose $\varepsilon\in (0,s-d/q)$.

\smallskip

(5) \textit{The case $\ell\le \iota$.} We are done once we have shown
\begin{equation}\label{eq:small-ell}
\bigg\Vert \mathbf{1}_{B_{3R}^{d_{\mathrm{CC}}}(0)} \sum_{\ell =-\ell_0}^\iota  F_\ell^{(\iota)}(L,U) f\bigg\Vert_p
 \lesssim  2^{-\varepsilon\iota} \norm{F^{(\iota)}}_{L^2_s} \norm{f}_p.
\end{equation}
We will use the rapid decay in $R$ from \cref{lem:rapid-decay} for this. Let $C>0$ be the constant from \cref{eq:C}, that is,
\[
B_R^{d_{\mathrm{CC}}}(0) \subseteq B_{CR}(0) \times B_{CR^2}(0).
\]
Given $\ell \in\{-\ell_0,\dots,\iota\}$, we split the Euclidean ball $B_{CR}(0) \times B_{CR^2}(0)$ into a grid with respect to the first layer, which yields a decomposition of $\supp f\subseteq B_R^{d_{\mathrm{CC}}}(0)$ such that
\[
\supp f = \bigcup_{m=1}^{M_{\ell}} B_{m}^{(\ell)},
\]
where $B_{m}^{(\ell)} \subseteq B_{CR_\ell}(x_{m}^{(\ell)}) \times B_{CR^2}(0)$ are disjoint subsets, and $|x_{m}^{(\ell)} - x_{m'}^{(\ell)}| > R_\ell/2$ for $m\neq m'$. Given $\gamma>0$, the number of overlapping balls
\[
\tilde B_{m}^{(\ell)}:=B_{2 C R_\ell R^\gamma}(x_{m}^{(\ell)}) \times B_{9 C R^2}(0),\quad 1\le m\le M_{\ell}
\] 
can be bounded by a constant $N_\gamma \lesssim_\iota 1$ (which is independent of $\ell$), where the notation
\begin{equation}\label{eq:iota-notation}
A\lesssim_\iota B
\end{equation}
means that $A\le R^{C(p,d_1,d_2)\gamma} B$ for some constant $C(p,d_1,d_2)>0$ depending only on the parameters $p,d_1,d_2$. We decompose the function $f$ as
\[
f = \sum_{m=1}^{M_{\ell}} f|_{B_{m}^{(\ell)}}.
\]
and write the left-hand side of \cref{eq:small-ell} as
\begin{align*}
\mathbf{1}_{B_{3R}^{d_{\mathrm{CC}}}(0)} & \sum_{\ell =-\ell_0}^\iota  F_\ell^{(\iota)}(L,U) f \\
& = \sum_{\ell=-\ell_0}^\iota \sum_{m=1}^{M_{\ell}} \Big(\mathbf{1}_{\tilde B_{m}^{(\ell)}}+\mathbf{1}_{\g\setminus\tilde B_{m}^{(\ell)}}\Big) \mathbf{1}_{B_{3R}^{d_{\mathrm{CC}}}(0)} F_\ell^{(\iota)}(L,U) (f|_{B_{m}^{(\ell)}}).
\end{align*}
Since $M_\ell\sim (R/R_\ell)^{d_1}\lesssim R^{d_1}$, the second summands are negligible thanks to the rapid decay from \cref{lem:rapid-decay} and the Sobolev embedding discussed in \cref{rem:sobolev}. Hence, it suffices to show
\begin{equation}\label{eq:main-term}
\bigg\Vert \sum_{\ell=-\ell_0}^\iota \sum_{m=1}^{M_{\ell}} \mathbf{1}_{\tilde B_m^{(\ell)}\cap B_{3R}^{d_{\mathrm{CC}}}(0)} F_\ell^{(\iota)}(L,U) (f|_{B_{m}^{(\ell)}}) \bigg\Vert_p
 \lesssim  2^{-\varepsilon\iota} \norm{F^{(\iota)}}_{L^2_s} \norm{f}_p.
\end{equation}
This means on a formal level that we may indeed assume that the convolution kernel $\mathcal K_\ell^{(\iota)}$ of $F_\ell^{(\iota)}(L,U)$ is supported in a ball of size $R_\ell R^\gamma\times R^2$.

As an abbreviation, we write
\[
g_{m}^{(\ell)} = \mathbf{1}_{\tilde B_m^{(\ell)}\cap B_{3R}^{d_{\mathrm{CC}}}(0)} F_\ell^{(\iota)}(L,U) (f|_{B_{m}^{(\ell)}}).
\]
Using Hölder's inequality for the sums over $\ell$ and $m$ and the bounded overlapping property of the balls $\tilde B_{m}^{(\ell)}$, we obtain
\begin{equation}\label{eq:main-term-overlap}
\bigg\Vert \sum_{\ell=-\ell_0}^\iota \sum_{m=1}^{M_{\ell}} g_{m}^{(\ell)} \bigg\Vert_p^p
 \lesssim_\iota (\iota+1+\ell_0)^{p-1} \sum_{\ell=-\ell_0}^\iota \sum_{m=1}^{M_{\ell}} \Vert g_{m}^{(\ell)}\Vert_p^p.
\end{equation}

\smallskip

(5.a) \textit{The case $(d_1,d_2)\notin\{(4,3),(8,6),(8,7)\}$.} Applying Hölder's inequality with $1/q=1/p-1/2$ in combination with the restriction type estimate \cref{eq:restriction-type-i} yields
\begin{align}
\Vert g_{m}^{(\ell)}\Vert_p
& \lesssim_\iota ( R_\ell^{d_1} R^{2d_2})^{1/q} \big\|F_\ell^{(\iota)}(L,U) (f|_{B_{m}^{(\ell)}})\big\|_2 \notag \\
& \lesssim (R_\ell^{ d_1-d_2}R^{2d_2})^{1/q} \norm{(F^{(\iota)}\psi)(\sqrt{\cdot}\,)}_{2^{\ell},2} \big\|(f|_{B_{m}^{(\ell)}})\big\|_p. \label{eq:appl-H-R}
\end{align}
Since $\ell\in \{-\ell_0,\dots,\iota\}$, \cref{eq:norm} yields
\begin{equation}\label{eq:sobolev-type}
\norm{(F^{(\iota)}\psi)(\sqrt{\cdot}\,)}_{2^{\ell},2}
\lesssim \norm{F^{(\iota)}\psi}_2 + 2^{-\ell \tilde s} \norm{F^{(\iota)}\psi}_{L^2_{\tilde s}} 
\lesssim_\psi \Big(\frac R {R_\ell} \Big)^{\tilde s} \norm{F^{(\iota)}}_2
\end{equation}
for all $\tilde s>1/2$. Plugging \cref{eq:appl-H-R} and \cref{eq:sobolev-type} into the right-hand side of \eqref{eq:main-term-overlap} and using the fact that the functions $f|_{B_{m}^{(\ell)}}$ have disjoint support, we get that the left-hand side of \cref{eq:main-term} can be dominated by a constant $C_\iota\lesssim_\iota 1$ times
\begin{equation}\label{eq:geom-series}
(\iota+1+\ell_0)^{p-1} \sum_{\ell=-\ell_0}^\iota \bigg(R_\ell^{ d_1-d_2}R^{2d_2} \Big(\frac R {R_\ell} \Big)^{\tilde s q} \bigg)^{p/q} \norm{F^{(\iota)}}_2^p \norm{f}_p^p.
\end{equation}
Note that
\[
\norm{F^{(\iota)}}_2 \sim 2^{-3\varepsilon\iota} R^{-d/q} \norm{F^{(\iota)}}_{L^2_s}
\]
for $\varepsilon=(s-d/q)/3$. Thus, \cref{eq:geom-series} is comparable to
\begin{equation}\label{eq:geom-series-ii}
2^{-\varepsilon p \iota} \big( 2^{-\varepsilon p \iota} (\iota+1)^{p-1}\big) \sum_{\ell=-\ell_0}^\iota \bigg( \Big(\frac {R_\ell} R \Big)^{d_1-d_2-\tilde s q} 2^{-\varepsilon q\iota } \bigg)^{p/q} \norm{F^{(\iota)}}_{L^2_s}^p \norm{f}_p^p.
\end{equation}
Choosing $\gamma>0$ small enough, we may conclude that the last term is bounded by a constant times $2^{-\varepsilon p \iota} \norm{F^{(\iota)}}_{L^2_s} \norm{f}_p$ if we find a proper bound for the sum
\[
\sum_{\ell=-\ell_0}^\iota \bigg( \Big(\frac {R_\ell} R \Big)^{d_1-d_2-\tilde s q} 2^{-\varepsilon q\iota } \bigg)^{p/q}.
\]
Recall that $R=2^\iota$ and $R_\ell=2^\ell$. Thus, substituting $\ell-\iota=-k$, we see that this sum is bounded by
\begin{equation}\label{eq:geom-series-iii}
\sum_{k=0}^\infty \big( 2^{-k(d_1-d_2-\tilde s q+\varepsilon q)} \big)^{p/q}.
\end{equation}
Since $1\le p\le p_{d_2}$, we have
\begin{equation}\label{eq:bound-q}
\frac 1 q = \frac 1 p - \frac 1 2 \ge \frac{d_2+3}{2(d_2+1)} - \frac 1 2 = \frac{1}{d_2+1}. 
\end{equation}
Suppose that $d_1 > 3d_2/2$. Since $d_1 - 3d_2/2\in \tfrac 1 2 \Z$, this means
\[
d_1 \ge \frac{3d_2} 2 + \frac 1 2.
\]
Thus, together with \cref{eq:bound-q}, we obtain
\begin{align*}
d_1-d_2-\tilde s q + \varepsilon q
 & \ge \frac{3d_2} 2 + \frac 1 2 - d_2 - \tilde s (d_2+1) + \varepsilon q \\
 & = \Big(\frac 1 2 - \tilde s \Big)(d_2+1) + \varepsilon q  > 0 ,
\end{align*}
if we choose $\tilde s>1/2$ sufficiently close to $1/2$. Then the geometric series in \cref{eq:geom-series-iii} is bounded by a constant independent of $\iota$, and we can bound \cref{eq:geom-series-ii} by a constant times $2^{-\varepsilon p\iota} \norm{F^{(\iota)}}_{L^2_s} \norm{f}_p$. Thus, by \cref{prop:Radon-Hurwitz-skew}, we are done with the proof except for the cases where $(d_1,d_2)\in\{(4,3),(8,6),(8,7)\}$.

\medskip

(5.b) \textit{The case $(d_1,d_2)\in\{(4,3),(8,6),(8,7)\}$.} We employ the restriction type estimate of \cref{thm:restriction-type}, which yields
\begin{equation}\label{eq:restriction-iii}
\norm{ F_\ell^{(\iota)}(L,U) }_{p\to 2}
\le C_{p,\chi} 2^{-\ell d_2(\frac 1 p - \frac 1 2)} \|(F^{(\iota)}\psi)(\sqrt{\cdot}\,)\|_2^{1-\theta_p} \norm{(F^{(\iota)}\psi)(\sqrt{\cdot}\,)}_{2^{\ell},2}^{\theta_p}
\end{equation}
where $\theta_p\in [0,1]$ satisfies $1/p = (1-\theta_p) + \theta_p/p_{d_2}$. (Note that $\min\{p_{d_1},p_{d_2}\}=p_{d_2}$.) By \cref{eq:sobolev-type}, the right-hand side of \cref{eq:restriction-iii} is bounded by a constant times
\[
2^{-\ell d_2(\frac 1 p - \frac 1 2)} \Big(\frac R {R_\ell} \Big)^{\theta_p \tilde s} \|F\|_2
\]
If we replace $\tilde s$ by $\theta_p\tilde s$ in the arguments above (keeping $\tilde s>1/2$), we see that we can conclude \eqref{eq:main-term} as long as we can bound, instead of \cref{eq:geom-series-iii}, the sum
\[
\sum_{\ell=0}^\infty \big( 2^{-\ell(d_1-d_2-\theta_p \tilde s q+\varepsilon q)} \big)^{p/q}.
\]
Similarly as above, by choosing $\tilde s>1/2$ sufficiently close to $1/2$, the above sum can be bounded from above if
\begin{equation}\label{eq:condition-ii}
d_1 - d_2 - \frac{\theta_p}{2} q \ge 0.
\end{equation}
Plugging in $\theta_p=p_{d_2}'(1-1/p)$ and $1/q=1/p-1/2$ (where $p_{d_2}'$ denotes the dual exponent of $p_{d_2}'$), one directly verifies that, for $1\le p\le 2$, \cref{eq:condition-ii} is equivalent to
\begin{equation}\label{eq:condition-iii}
1\le p \le  \frac{p_{d_2}'+2(d_1-d_2)}{p_{d_2}'+d_1-d_2}.
\end{equation}
Since
\[
p_{d_2}' = \frac{2\left(d_2+1\right)}{d_2-1} = 
\begin{cases}
4 & \text{if }  d_2 =3 \\
14/5 & \text{if }d_2 =6 \\
8/3 & \text{if }d_2 =7,
\end{cases}
\]
the right-hand side of \cref{eq:condition-iii} reads 
\[
\frac{p_{d_2}'+2(d_1-d_2)}{p_{d_2}'+d_1-d_2}
= 
\begin{cases}
6/5 & \text{if }(d_1,d_2) =(4,3) \\
17/12 & \text{if }(d_1,d_2) =(8,6) \\
14/11 & \text{if }(d_1,d_2) =(8,7),
\end{cases}
\]
which concludes the proof of \cref{prop:reduced}.
\end{proof}

\section{\texorpdfstring{Improvement for the case $(d_1,d_2)=(4,3)$}{Improvement for the case (4,3)}} \label{sec:4-3}

In this section, we treat the case where $(d_1,d_2)=(4,3)$ in the stratification $\g=\g_1\oplus\g_2\cong \R^{d_1}\oplus \R^{d_2}$, still working under the assumption that the underlying Lie group $G$ is a Métivier group. The goal of this section is to show the following enhancement of \cref{prop:reduced} in the case $(d_1,d_2)=(4,3)$, which in turn implies the spectral multiplier estimates of \cref{thm:multiplier} for the stated range $1\le p\le p_{3}=4/3$ via \cref{cor:reduction}.

\begin{proposition}\label{prop:4-3}
Let $G$ be a Métivier group with $(d_1,d_2)=(4,3)$. Suppose that $1\le p\le p_{3}=4/3$. If $s>d\left( 1/p - 1/2\right)=7\left( 1/p - 1/2\right)$, then there exists some $\varepsilon>0$ such that
\[
\Vert F^{(\iota)}(\sqrt L) \Vert_{p\to p} \le C_{p,s} 2^{-\varepsilon\iota} \norm{F^{(\iota)}}_{L^2_s} \quad\text{for all } \iota\in\N
\]
and any even bounded Borel function $F\in L^2_s$ supported in $[-2,-1/2]\cup[1/2,2]$.
\end{proposition}

Additionally to the first layer weighted Plancherel estimate \cref{eq:weighted-plancherel-1} we need the following Plancherel estimate with weight on the second layer.

\begin{proposition}\label{prop:weighted-plancherel-2}
Let $G$ be a Métivier group with layers of dimensions $d_1=4$ and $d_2=3$. Then there is a direction $v\in \g_2$ with $|v|=1$ such that, if $F:\R\to\C$ is a smooth function supported in a compact subset $A\subseteq (0,\infty)$ and $\chi:(0,\infty)\to\C$ is some bump function, the convolution kernel $\mathcal K_{F(L)\chi(2^\ell U)}$ of $F(L)\chi(2^\ell U)$ satisfies
\begin{equation}\label{eq:weighted-plancherel-2}
\int_G \big| \langle u,v\rangle^\alpha \mathcal K_{F(L)\chi(2^\ell U)}(x,u)\big|^2 \,d(x,u) \le C_{A,\chi,\alpha}\, 2^{\ell(2\alpha-d_2)} \norm{F}_{L^2_\alpha}^2
\end{equation}
for all $\alpha\ge 0$ and $\ell\in\Z$.
\end{proposition}

We first show how the spectral multiplier estimate of \cref{prop:4-3} is derived from the weighted Plancherel estimate of \cref{prop:weighted-plancherel-2}.

\subsection{Proof of \texorpdfstring{\cref{prop:4-3}}{Proposition 10.1}}

In the proof of \cref{prop:reduced}, we used the weighted Plancherel estimates of \cref{prop:weighted-plancherel-1} to show that the convolution kernel
\begin{align*}
\mathcal K^{(\iota)}_\ell(x,u) = (2\pi)^{-d_2} \sum_{\mathbf{k}\in\N^N} \int_{\g_{2,r}^*} (F^{(\iota)}\psi)\big(\sqrt{\eigvp{k}{\mu}}\big) \chi(2^\ell |\mu|) \\
\times \prod_{n=1}^N \varphi_{k_n}^{(b_n^\mu,r_n)}(R_\mu^{-1} P_n^\mu x)\, e^{i\langle \mu, u\rangle} \, d\mu
\end{align*}
is essentially supported in a ball of size $R_\ell\times R^2$. On the other hand, if $\omega=\mu/|\mu|$, then $\eigvp{k}{\mu}=\left|\mu\right|\lambda_{\mathbf{k}}^{\omega}$, where $\lambda_{\mathbf{k}}^{\omega}\sim 2^\ell=R_\ell$ on the support of the above integrand. Thus, on a heuristic level, when thinking of using integration by parts in the above integral, the scaling of $\mu$ by $\lambda_{\mathbf{k}}^{\omega}\sim R_\ell$ suggests that one might hope that the convolution kernel is even supported in a ball of size $R_\ell \times 2^\ell R$ in place of $R_\ell \times R^2$, up to rapid decay. Since derivatives in $\mu$ correspond directly to weights in terms of $u$ via the Fourier transform, weighted Plancherel estimates are again a good tool to make this idea precise. The following lemma shows that the discussed phenomenon can at least be observed along the direction $v$ if the weighted Plancherel estimate \cref{eq:weighted-plancherel-2} is available.

\begin{lemma}\label{lem:rapid-decay-ii}
Suppose that $1\le p< 2$. Consider the radii $R=2^\iota$ and $R_\ell=2^\ell$ for $\iota\in\N$ and $\ell \in\{-\ell_0,\dots,\iota\}$. Suppose that $F:\R\to\C$ is an even bounded Borel function supported in $[-2,-1/2]\cup [1/2,2]$ and $\psi,\chi:(0,\infty)\to\C$ are bump functions. Let $\gamma>0$. Given $C,\tilde C>0$, consider the balls
\begin{align*}
A^{(\ell)} & = B_{C R_\ell}(0) \times B_{\tilde C R^2}(0) \subseteq  \g_1 \times \g_2 \quad\text{and} \\
\tilde A^{(\ell)} & = B_{2C R_\ell R^\gamma}(0) \times B_{9\tilde C R^{2+\gamma}}(0) \subseteq \g_1 \times \g_2.
\end{align*}
Let $v\in\g_2$ with $|v|=1$ be as in \cref{prop:weighted-plancherel-2} and $u^{(\ell)}_0\in\g_2$ be some point lying on the line along $v$.
Then there is a constant $\bar C>0$ depending only on $C,\tilde C$ and the Lie group~$G$ such that, if $f\in L^p(G)\cap L^2(G)$ is a function supported in the set
\[
A^{(\ell)}(u^{(\ell)}_0) = \{ (x,u) \in A^{(\ell)} : |\langle u- u^{(\ell)}_0,v\rangle| < \tilde C R_\ell R \},
\]
then the $L^p$-norm of $\mathbf{1}_{\tilde A^{(\ell)}} (F^{(\iota)}\psi)(\sqrt L) \chi(2^\ell U) f$ is rapidly decaying in $R$ outside the set
\[
\tilde A^{(\ell)}(u^{(\ell)}_0)  :=\{ (x,u) \in \tilde A^{(\ell)} : |\langle u- u^{(\ell)}_0,v\rangle| < 2\bar C R_\ell R^{1+\gamma} \}.
\]
That is, for any $N\in\N$, there is a constant $C_{N,\gamma,\psi,\chi}>0$ independent of $u^{(\ell)}_0$ such that
\begin{equation}\label{eq:rapid-decay-ii}
\big\|\mathbf{1}_{(\g\setminus\tilde A^{(\ell)}_0)\cap\tilde A^{(\ell)}} (F^{(\iota)}\psi)(\sqrt L) \chi(2^\ell U) f \big\|_p \le C_{N,\gamma,\psi,\chi} R^{-N} \|F^{(\iota)}\|_\infty \|f\|_p.
\end{equation}
\end{lemma}

\begin{remark}
Note that $\|F^{(\iota)}\|_\infty \lesssim \smash{\|F^{(\iota)}\|_{L^2_\alpha}} \sim R^\alpha \|F^{(\iota)}\|_2$ for $\alpha>1/2$. Hence, due to the rapid decay in $R$ in \cref{eq:rapid-decay-ii}, we may also replace $\|F^{(\iota)}\|_\infty$ by $\|F^{(\iota)}\|_2$ on the right-hand side of \cref{eq:rapid-decay-ii}.
\end{remark}

\begin{proof}
Similar to \cref{lem:rapid-decay}, we interpolate between an $L^1$ and an $L^2$-estimate. For the $L^2$-estimate, we again just use that
\begin{equation}\label{eq:interpolation-L2}
\big\|\mathbf{1}_{(\g\setminus\tilde A^{(\ell)}(u_0^{(\ell)}))\cap\tilde A^{(\ell)}} (F^{(\iota)}\psi)(\sqrt L) \chi(2^\ell U) f \big\|_2
\lesssim_{\psi,\chi} \|F^{(\iota)}\|_\infty \|f\|_2.
\end{equation}
For the $L^1$-estimate, we exploit the weighted Plancherel estimate \cref{eq:weighted-plancherel-2}. We consider the sets
\[
\bar A_\ell(u') :=\{ (x,u) \in \tilde A^{(\ell)} : |\langle u-u',v\rangle| \ge \bar C R_\ell R^{1+\gamma} \},\quad u'\in \g_2.
\]
Then, if the points $(x,u),(x',u')\in G$ satisfy
\[
(x,u)\in (\g\setminus\tilde A^{(\ell)}(u_0^{(\ell)}))\cap\tilde A^{(\ell)}
\quad\text{and}\quad
(x',u')\in \supp f,
\]
we have
\[
|\langle u-u_0^{(\ell)},v \rangle| \ge 2\bar C R_\ell R^{1+\gamma} \quad\text{and}\quad |\langle u'-u_0^{(\ell)},v\rangle | < \tilde C R_\ell R,
\]
and, if we choose $\bar C\ge \tilde C$, in particular $(x,u)\in \bar A_\ell(u')$.
Using again \cref{prop:conv-kernel}, we write
\begin{align*}
 & (F^{(\iota)}\psi)(\sqrt L)\chi(2^\ell U) f(x,u)
 = f * \mathcal K_\ell^{(\iota)}(x,u) \\
& = \int_G f(x',u') \mathcal K_\ell^{(\iota)}\big((x',u')^{-1}(x,u)\big)\,d(x',u'),\quad (x,u)\in G.
\end{align*}
Hence, by Fubini's theorem,
\begin{align*}
& \big\|\mathbf{1}_{(\g\setminus\tilde A^{(\ell)}(u_0^{(\ell)}))\cap\tilde A^{(\ell)}} (F^{(\iota)}\psi)(\sqrt L) \chi(2^\ell U) f \big\|_1 \\
& \le \int_G \mathbf{1}_{(\g\setminus\tilde A^{(\ell)}(u_0^{(\ell)}))\cap\tilde A^{(\ell)}}(x,u)  |f * \mathcal K_\ell^{(\iota)}(x,u)| \,d(x,u) \\
& \le \int_{\supp f} |f(x',u')| \int_{\bar A_\ell(u')} \big|\mathcal K_\ell^{(\iota)}\big((x',u')^{-1}(x,u)\big)\big|\,d(x,u)\,d(x',u')\\
& \le \norm{f}_1 \sup_{(x',u')\in A^{(\ell)}} \int_{\bar A_\ell(u')} \big|\mathcal K_\ell^{(\iota)}\big((x',u')^{-1}(x,u)\big)\big|\,d(x,u).
\end{align*}
Note that
\[
(x',u')^{-1}(x,u)=(x-x',u-u'-\tfrac 1 2 [x',x]).
\]
If $(x',u')\in A^{(\ell)}$ and $(x,u)\in \bar A_\ell(u')$, we have
\[
|x'|< CR_\ell\quad\text{and}\quad |x|< 2CR_\ell R^\gamma,
\]
and thus, if we choose the constant $\bar C>0$ large enough,
\begin{align*}
|\langle u-u'-\tfrac 1 2 [x',x],v \rangle|
& \ge |\langle u-u',v \rangle| - \tfrac 1 2 C_0 |x'| |x|\\
& \ge \bar C R_\ell R^{1+\gamma} - C_0 C^2 R_\ell^2 R^\gamma \gtrsim  R_\ell R^{1+\gamma},
\end{align*}
where $C_0>0$ is the constant given by
\[
C_0:=\sup_{x,x'\neq 0} \frac{\left|[x,x']\right|}{|x||x'|}.
\]
Hence, if $(x',u')\in A^{(\ell)}$, then
\begin{align}
\int_{\bar A^{(\ell)}(u')} & \big|\mathcal K_\ell^{(\iota)}\big((x',u')^{-1}(x,u)\big)\big|\,d(x,u) \notag \\
& \lesssim (R_\ell R^{1+\gamma})^{-N} \int_{\tilde A^{(\ell)}} |\langle u-u'-\tfrac 1 2 [x',x],v\rangle|^N \notag \\
& \hspace{3cm} \times |\mathcal K_\ell^{(\iota)}(x-x',u-u'-\tfrac 1 2 [x',x])|\,d(x,u)  \notag \\
& \le (R_\ell R^{1+\gamma})^{-N} | \tilde A^{(\ell)}|^{1/2} \bigg(\int_G \big||\langle u,v\rangle|^N \mathcal K_\ell^{(\iota)}(x,u)\big|^2 \,d(x,u)\bigg)^{1/2}. \label{eq:apply-weighted-1}
\end{align}
Now the weighted Plancherel estimate \cref{eq:weighted-plancherel-2} yields
\[
\bigg(\int_G \big||\langle u,v\rangle|^N \mathcal K_\ell(x,u)\big|^2 \,d(x,u)\bigg)^{1/2} \lesssim_N R_\ell^{N-d_2/2} \norm{F^{(\iota)}}_{L^2_N}.
\]
Hence, the last line of \cref{eq:apply-weighted-1} is bounded by a constant times
\begin{equation}\label{eq:apply-weighted-2}
(R^{1+\gamma})^{-N} |\tilde A^{(\ell)}|^{1/2} R_\ell^{-d_2/2} \norm{F^{(\iota)}}_{L^2_N}.
\end{equation}
We use again the notation $\lesssim_\iota$ introduced in \eqref{eq:iota-notation}.
Since $|\tilde A^{(\ell)}|\lesssim_\iota R_\ell^{d_1}R^{2d_2}$ and $\norm{F^{(\iota)}}_{L^2_N}\sim R^N \norm{F^{(\iota)}}_2$, \cref{eq:apply-weighted-2} can be bounded by a constant $C_\iota\lesssim_\iota 1$ times
\[
R^{-\gamma N} R_\ell^{(d_1-d_2)/2} R^{d_2} \norm{F^{(\iota)}}_2.
\]
In particular, choosing $N=N(\gamma)$ large enough, we obtain
\begin{equation}\label{eq:interpolation-L1}
\big\|\mathbf{1}_{(\g\setminus\tilde A^{(\ell)}(u_0^{(\ell)}))\cap\tilde A^{(\ell)}} (F^{(\iota)}\psi)(\sqrt L) \chi(2^\ell U) f \big\|_1
 \lesssim_{N} R^{-N} \|F^{(\iota)}\|_2 \|f\|_1.
\end{equation}
Interpolation between the bounds \eqref{eq:interpolation-L1} and \eqref{eq:interpolation-L2} yields \cref{eq:rapid-decay-ii}.
\end{proof}

With the rapid decay from \cref{lem:rapid-decay-ii} at hand, we can now prove the spectral multiplier estimates of \cref{prop:4-3} under the assumption of \cref{prop:weighted-plancherel-2}.

\begin{proof}[Proof of \cref{prop:4-3}]
Let $R=2^\iota$ and $R_\ell=2^\ell$ for $\ell\in \{-\ell_0,\dots,\iota\}$, where $\ell_0$ is defined as in \cref{rem:ell_0}. Arguing as in the proof of \cref{prop:reduced}, it suffices to show \eqref{eq:main-term}, that is,
\begin{equation}\label{eq:main-term-b}
\bigg\Vert \sum_{\ell=-\ell_0}^\iota \sum_{m=1}^{M_{\ell}} \mathbf{1}_{\tilde B_m^{(\ell)}\cap B_{3R}^{d_{\mathrm{CC}}}(0)} F_\ell^{(\iota)}(L,U) (f|_{B_{m}^{(\ell)}}) \bigg\Vert_p
 \lesssim  2^{-\varepsilon\iota} \norm{F^{(\iota)}}_{L^2_s} \norm{f}_p,
\end{equation}
where $f|_{B_{m}^{(\ell)}}$ is the restriction of $f$ to the set
\[
B_{m}^{(\ell)} \subseteq B_{CR_\ell}(x_{m}^{(\ell)}) \times B_{CR^2}(0)\quad\text{with } |x_m^{(\ell)}|\le CR,
\]
and $F_\ell^{(\iota)}(L,U)(f|_{B_{m}^{(\ell)}})$ is restricted to
\[
\tilde B_{m}^{(\ell)}=B_{2C R_\ell R^\gamma}(x_{m}^{(\ell)}) \times B_{9 C R^2}(0).
\]
To prove \cref{eq:main-term-b}, it suffices to show
\begin{equation}\label{eq:main-term-c}
\big\| \mathbf{1}_{\tilde B_m^{(\ell)}\cap B_{3R}^{d_{\mathrm{CC}}}(0)} F_\ell^{(\iota)}(L,U) (f|_{B_{m}^{(\ell)}}) \big\|_p
 \lesssim  2^{-\varepsilon\iota} \norm{F^{(\iota)}}_{L^2_s} \big\|f|_{B_{m}^{(\ell)}}\big\|_p.
\end{equation}
Indeed, thanks to the bounded overlap of the balls $\tilde B_{m}^{(\ell)}$, given \cref{eq:main-term-c}, we obtain
\begin{align*}
& \bigg\| \sum_{\ell=-\ell_0}^\iota \sum_{m=1}^{M_{\ell}} \mathbf{1}_{\tilde B_m^{(\ell)}\cap B_{3R}^{d_{\mathrm{CC}}}(0)} F_\ell^{(\iota)}(L,U) (f|_{B_{m}^{(\ell)}}) \bigg\|_p^p \\
& \lesssim_\iota (\iota+1+\ell_0)^{p-1}  \sum_{\ell=-\ell_0}^\iota \sum_{m=1}^{M_{\ell}} \big\|  \mathbf{1}_{\tilde B_m^{(\ell)}\cap B_{3R}^{d_{\mathrm{CC}}}(0)} F_\ell^{(\iota)}(L,U) (f|_{B_{m}^{(\ell)}}) \big\|_p^p\\
& \lesssim (\iota+1+\ell_0)^{p-1}  \sum_{\ell=-\ell_0}^\iota \sum_{m=1}^{M_{\ell}} \big(2^{-\varepsilon\iota} \norm{F^{(\iota)}}_{L^2_s} \big\|f|_{B_{m}^{(\ell)}}\big\|_p\big)^p \\
& \lesssim \big( 2^{-\tilde \varepsilon\iota} \norm{F^{(\iota)}}_{L^2_s} \|f\|_p\big)^p\qquad \text{for } 0<\tilde\varepsilon<\varepsilon,
\end{align*}
where we use again the notation $\lesssim_\iota$ introduced in \eqref{eq:iota-notation}.

Let $\lambda_{(x_{m}^{(\ell)},0)}$ denote the left multiplication by $(x_{m}^{(\ell)},0)$ on $G$. Note that
\[
\lambda_{(x_{m}^{(\ell)},0)}(x,u)\in B_{m}^{(\ell)}
\]
implies $|x|\le C R_\ell$ and $\big|u+\tfrac 1 2 [x_{m}^{(\ell)},x]\big| \le C R^2$, and thus in particular
\[
|u| \lesssim R^2 + |x_{m}^{(\ell)}|\left|x\right| \lesssim R^2,
\]
meaning that $(x,u)$ is contained in the ball
\[
A^{(\ell)} := B_{C R_\ell}(0) \times B_{\tilde C R^2}(0),
\]
with $\tilde C>0$ chosen sufficiently large.
Similarly,
\[
\lambda_{(x_{m}^{(\ell)},0)}(x,u)\in \tilde B_{m}^{(\ell)}
\]
implies that $(x,u)$ lies in the ball 
\[
\tilde A^{(\ell)} := B_{2C R_\ell R^\gamma}(0) \times B_{9\tilde C R^{2+\gamma}}(0).
\]
Since the operator $F_\ell^{(\iota)}(L,U)$ is left-invariant, we have
\[
\big(F_\ell^{(\iota)}(L,U) (f|_{B_{m}^{(\ell)}})\big)(\lambda_{(x_{m}^{(\ell)},0)}(x,u)) = 
F_\ell^{(\iota)}(L,U) \tilde f_{m}^{(\ell)}(x,u),
\]
where $\tilde f_{m}^{(\ell)}(x,u) := f_{m}^{(\ell)}(\lambda_{(x_{m}^{(\ell)},0)}(x,u))$ is supported in the ball $A^{(\ell)}$. Hence, since
\[
\big\|\mathbf{1}_{\tilde B_m^{(\ell)}} F_\ell^{(\iota)}(L,U) (f|_{B_{m}^{(\ell)}}) \big\|_p
 = \|\mathbf{1}_{\tilde A^{(\ell)}} F_\ell^{(\iota)}(L,U) \tilde f_{m}^{(\ell)} \|_p,
\]
for proving \cref{eq:main-term-c}, it suffices to show
\begin{equation}\label{eq:main-term-d}
\norm{\mathbf{1}_{\tilde A^{(\ell)}} F_\ell^{(\iota)}(L,U) f}_p \lesssim 2^{-\varepsilon\iota} \norm{F^{(\iota)}}_{L^2_s}\norm{ f}_p
\end{equation}
for functions $f$ that are supported in $A^{(\ell)}$. 

Let $v\in\g_2$ with $|v|=1$ be the direction of \cref{prop:weighted-plancherel-2} along which the weighted Plancherel estimate \cref{eq:weighted-plancherel-2} holds. We decompose the second layer of $A^{(\ell)}$ into slabs of thickness $\sim R_\ell R$ aligned with $v$. More precisely, there are disjoint subsets $A^{(\ell)}_k$ such that
\[
A^{(\ell)} = \bigcup_{k=1}^{K_\ell} A^{(\ell)}_k
\]
and points $u^{(\ell)}_k\in\g_2$ (thought of as centers) that satisfy the following properties:
\begin{itemize}
\item $u^{(\ell)}_k$ lies on the line along $v$ for all $k\in\{1,\dots,K_\ell\}$,
\item $|u^{(\ell)}_k - u^{(\ell)}_{k'}| > R_\ell R/2$ for $k\neq k'$,
\item $\langle u- u^{(\ell)}_k,v\rangle < \tilde C R_\ell R$ for all $u\in A^{(\ell)}_k$ and all $k\in\{1,\dots,K_\ell\}$.
\end{itemize}
Moreover, we consider the slightly enlarged slabs
\[
\tilde A^{(\ell)}_k :=\{ (x,u) \in \tilde A^{(\ell)} : \langle u- u^{(\ell)}_k,v\rangle < 2\bar C R_\ell R^{1+\gamma} \},
\]
with $\bar C$ being the constant from \cref{lem:rapid-decay-ii}.
With this setup, the maximal number $N_\gamma$ of overlapping sets $\tilde A^{(\ell)}_k$ is bounded by $N_\gamma \lesssim_\iota 1$. Given a function $f$ supported in $A^{(\ell)}$, we decompose it as
\[
f = \sum_{k=1}^{K_\ell} f|_{A^{(\ell)}_k}.
\]
Moreover, we decompose the function $h_k^{(\ell)}:=\mathbf{1}_{\tilde A^{(\ell)}} F_\ell^{(\iota)}(L,U) ( f|_{A^{(\ell)}_k})$ as
\[
h_k^{(\ell)} = \mathbf{1}_{\tilde A^{(\ell)}_k} h_k^{(\ell)} + \mathbf{1}_{\g\setminus\tilde A^{(\ell)}_k} h_k^{(\ell)}.
\]
Since $K_\ell\lesssim 2^{\iota-\ell}\le 2^{\iota}$, the rapid decay from \cref{lem:rapid-decay-ii} yields
\[
\bigg\Vert \sum_{k=1}^{K_{\ell}} \mathbf{1}_{\g\setminus\tilde A^{(\ell)}_k} h_k^{(\ell)} \bigg\Vert_p
 \lesssim_N R^{-N} \norm{F^{(\iota)}}_2 \norm{f}_p,\quad N>0.
\]
Hence, to prove \cref{eq:main-term-d}, it remains to show
\[
\bigg\Vert \sum_{k=1}^{K_{\ell}}  \mathbf{1}_{\tilde A^{(\ell)}_k} h_k^{(\ell)} \bigg\Vert_p \lesssim 2^{-\varepsilon\iota} \norm{F^{(\iota)}}_{L^2_s} \norm{f}_p.
\]
Due to the bounded overlap of the balls $\tilde A^{(\ell)}_k$, we have
\[
\bigg\Vert \sum_{k=1}^{K_{\ell}}  \mathbf{1}_{\tilde A^{(\ell)}_k} h_k^{(\ell)} \bigg\Vert_p
\lesssim_\iota 
\bigg( \sum_{k=1}^{K_{\ell}} {\norm{ \mathbf{1}_{\tilde A^{(\ell)}_k} h_k^{(\ell)}}_p^p} \bigg)^{1/p}.
\]
Let $1/q=1/p-1/2$.
Since $s>d/q$ and $\norm{F^{(\iota)}}_{L^2_{d/q}}\sim 2^{d/q-s}\norm{F^{(\iota)}}_{L^2_s}$, it suffices to show
\begin{equation}\label{eq:error-main-ii}
\bigg(\sum_{k=1}^{K_{\ell}} {\norm{ \mathbf{1}_{\tilde A^{(\ell)}_k} h_k^{(\ell)}}_p^p} \bigg)^{1/p}
 \lesssim_\iota \norm{F^{(\iota)}}_{L^2_{d/q}} \norm{f}_p.
\end{equation}
Now, using Hölder's inequality in combination with the restriction type estimate of \cref{thm:restriction-type}, we get
\begin{align*}
\norm{\mathbf{1}_{\tilde A^{(\ell)}_k} h_k^{(\ell)}}_p
& \lesssim_\iota ( R_\ell^{d_1+1} R^{2d_2-1})^{1/q} \norm{F_\ell^{(\iota)}(L,U) ( f|_{A^{(\ell)}_k})}_2 \\
& \lesssim \bigg(R_\ell^{d_1-d_2} R^{2 d_2} \frac{R_\ell}{R} \bigg)^{1/q} \norm{(F^{(\iota)}\psi)(\sqrt\cdot\,)}_{2^{\ell},2} \norm{ f|_{A^{(\ell)}_k}}_p
\end{align*}
for all $\tilde s>1/2$.
By the estimate \cref{eq:norm} for the norm $\|\cdot\|_{M,2}$, we have
\[
\norm{(F^{(\iota)}\psi)(\sqrt\cdot\,)}_{2^{\ell},2}
\lesssim \norm{F^{(\iota)}\psi}_2 + 2^{-\ell \tilde s} \norm{F^{(\iota)}\psi}_{L^2_{\tilde s}}
\lesssim_\psi \Big(\frac{R_\ell}{R}\Big)^{-\tilde s} \norm{F^{(\iota)}}_2.
\]
Using $R^{d/q} \norm{F^{(\iota)}}_2\sim \norm{F^{(\iota)}}_{L^2_{d/q}}$, we obtain
\begin{align}
\sum_{k=1}^{K_{\ell}} \Vert \mathbf{1}_{\tilde A^{(\ell)}_k} h_k^{(\ell)}\Vert_p^p
& \lesssim_\iota \sum_{k=1}^{K_{\ell}} \bigg(R_\ell^{d_1-d_2} R^{2d_2}\Big(\frac{R_\ell}{R}\Big)^{-\tilde s q+1} R^{-d}\bigg)^{p/q} \norm{F^{(\iota)}}_{L^2_{d/q}}^p \norm{f|_{A^{(\ell)}_k}}_p^p \notag \\
& = \bigg(\Big(\frac{R_\ell}{R}\Big)^{d_1-d_2-\tilde sq+1}\bigg)^{p/q} \norm{F^{(\iota)}}_{L^2_{d/q}}^p \norm{f}_p^p. \label{eq:condition-enhanced}
\end{align}
Using \eqref{eq:bound-q} and $(d_1,d_2)=(4,3)$, we have
\[
d_1-d_2-\tilde s q + 1 \ge d_1-d_2-\tilde s (d_2+1) + 1 = 2 - 4 \tilde s.
\]
Since the last term gets arbitrarily close to zero if we choose $\tilde s> 1/2$ sufficiently close to $1/2$, we obtain \cref{eq:error-main-ii}.
\end{proof}

\begin{remark}\label{rem:extension-H-type}
Having even a second layer weighted Plancherel estimate of the form
\begin{equation}\label{eq:full-plancherel}
\int_G \big| |u|^\alpha \mathcal K_{F(L)\chi(2^\ell U)}(x,u)\big|^2 \,d(x,u) \le C_{\alpha,\chi}\, 2^{\ell(2\alpha-d_2)} \norm{F}_{L^2_\alpha}^2
\end{equation}
in place of \cref{eq:weighted-plancherel-2} available, one could adapt the above arguments and pass to even smaller balls of size $R_\ell \times R_\ell R$ in $\g_1\times \g_2$ instead of just slicing the second layer $\g_2$ into slabs of thickness $R_\ell R$ along the direction $v$. This adaptation would raise the exponent $d_1-d_2-\tilde s q +1$ in \cref{eq:condition-enhanced} to $d_1-\tilde s q$, which is always positive if $\tilde s>1/2$ is chosen sufficiently close to $1/2$, since
\[
d_1-\tilde s q \ge d_1 - \tilde s (d_2+1)\quad \text{and} \quad d_1>d_2.
\]
For instance, Plancherel estimates of the form \cref{eq:full-plancherel} are known to hold for the Heisenberg--Reiter type groups considered in \cite{Ma15}, which in particular include direct products of Heisenberg type groups. Thus, using directly the estimates of \cite{Ma15}, one could prove the spectral multiplier estimates of \cref{thm:multiplier} for direct products of Heisenberg type groups in the range $1\le p\le 2(d_2+1)/(d_2+3)$ by adapting the proof of \cref{prop:4-3}, which would especially improve the range of $p$ for the cases $(d_1,d_2)=(8,6)$ and $(d_1,d_2)=(8,7)$.
\end{remark}

\subsection{Proof of the weighted Plancherel estimate}

It remains to show the  weighted Plancherel estimate \cref{eq:weighted-plancherel-2} of \cref{prop:weighted-plancherel-2}, that is,
\[
\int_G \big| \langle u,v\rangle^\alpha \mathcal K_{F(L)\chi(2^\ell U)}(x,u)\big|^2 \,d(x,u) \le C_{A,\chi,\alpha}\, 2^{\ell(2\alpha-d_2)} \norm{F}_{L^2_\alpha}^2
\]
for all $\alpha\ge 0$ and $\ell\in\Z$. To that end, we use the following lemma, which is a special case of Proposition~19 in \cite{MaMue14b}. We apply the lemma in the case where $G$ is a Métivier group with layers of dimensions $d_1=4$ and $d_2=3$, but the statement actually holds for any Métivier group. Note that the subsequent proof of \cref{prop:weighted-plancherel-2} also works for any Métivier group. The crucial point is the availability of the bounds \cref{eq:mu-bounds} assumed in \cref{lem:GAFA-adapted}.

We use again the notation of \cref{prop:rotation}. In particular, $\g_{2,r}^*\subseteq\g_2$ is the homogeneous Zariski-open subset of \cref{prop:rotation}, $b_1^\mu,\dots,b_N^\mu$ are the non-negative eigenvalues of the matrix $iJ_\mu$, and $P_1^\mu,\dots,P_N^\mu$ are the corresponding spectral projections of \cref{eq:spectral-decomp-0}.

\begin{lemma}\label{lem:GAFA-adapted}
Let $D$ be a smooth vector field on $\g_{2,r}^*$, thought of as a first-order differential operator in the variable $\mu \in \g_{2,r}^*$.
Suppose that there exist $\kappa >0$ and $\alpha_0\in\N$ such that
\begin{equation}\label{eq:mu-bounds}
\begin{gathered}
|D^\alpha b_1^\mu| \leq \kappa b_1^\mu, \dots, |D^\alpha b_N^\mu| \leq \kappa b_N^\mu, \\
\|D^\alpha P_1^\mu\| \leq \kappa,\dots,\|D^\alpha P_N^\mu\| \leq \kappa
\end{gathered}
\end{equation}
for all $\mu \in \g_{2,r}^*$ and all $\alpha\in\{0,\dots,\alpha_0\}$. Suppose further that $F:\R\to\C$ and $\psi:\R^{d_2}\to\C$ are bump functions. Let $V_{F\otimes \psi}$ defined be as in \cref{eq:V-tilde}, that is,
\[
V_{F\otimes \psi}(\xi,\mu) = \sum_{k \in \N^N} F(|P^\mu_0 \xi|^2+\eigvp{k}{\mu})\psi(\mu) \prod_{n=1}^N \mathcal L_{k_n}^{(r_n-1)}(|P^\mu_n \xi|^2 /b^\mu_n)
\]
with eigenvalues $\eigvp{k}{\mu}=\sum_{n=1}^N \left(2k_n+r_n\right)b_n^\mu$ and Laguerre functions $\mathcal L_{k_n}^{(r_n-1)}$. Then
\begin{multline*}
\int_{\g_1} |D^\alpha V_{F\otimes \psi}(\xi,\mu)|^2 \, d\xi 
\leq C_{\kappa,\alpha} \sum_{\iota \in I_{\alpha}} \int_{R_\iota} \sum_{\mathbf{k}\in \N^N} |D^{\delta_\iota} \psi(\mu)|^2 \\
\times \bigg| F^{(\gamma_\iota)}\bigg(\sum_{n=1}^N \left(2(k_n+s_n)+r_n\right) b_n^\mu \bigg)\bigg|^2 \,
\bigg| \prod_{n=1}^N (b^\mu_n)^{1+a^\iota_n} (k_n+1)^{a^\iota_n} \bigg| \,d\nu_{\iota}(s),
\end{multline*}
for all $\mu \in \g_{2,r}^*$ and all $\alpha\in \{0,\dots,\alpha_0\}$, where $I_{\alpha}$ is a finite set and, for all $\iota \in I_{\alpha}$,
\begin{itemize}
\item $a^\iota \in \N^N$, $\gamma_\iota,\delta_\iota \in \N$, $\gamma_\iota + \delta_\iota  \leq \alpha$,
\item $R_\iota = \prod_{n=1}^N [ 0,\beta^\iota_n ]$ and $\nu_\iota$ is a Borel probability measure on $R_\iota$,
\end{itemize}
and where $F^{(\gamma_\iota)}$ denotes the $\gamma_\iota$-th derivative of $F$.
\end{lemma}

We will show that the bounds of \eqref{eq:mu-bounds} hold with
\[
D = \left|\mu\right|\partial_v
\]
for some directional derivative $\partial_v$ on $\g_2^*$.
Before verifying these bounds, we show how the weighted Plancherel estimate \eqref{eq:weighted-plancherel-2} follows from \cref{lem:GAFA-adapted}, which is an adaptation of the arguments of \cite[Proof of Proposition 15]{MaMue16}.

\begin{proof}[Proof of \cref{prop:weighted-plancherel-2}]
Suppose that the bounds \cref{eq:mu-bounds} hold with $D = \left|\mu\right|\partial_v$. We apply \cref{lem:GAFA-adapted} with $\psi(\mu)=\chi(2^\ell|\mu|)$. In view of \cref{eq:conv-kernel-alt}, the Plancherel theorem yields
\begin{align}\label{eq:Plancherel-V}
& \int_G | \langle u,v\rangle^\alpha \mathcal K_{F(L)\chi(2^\ell U)}(x,u) |^2 \, d(x,u) \notag \\
& \sim \int_{\g_{2,r}^*} \int_{\g_1} | \partial_v^\alpha V_{F\otimes \psi}(\xi,\mu)|^2\,d\xi \, d\mu \notag \\
& = \int_{\g_{2,r}^*} |\mu|^{-2\alpha} \int_{\g_1} | D^\alpha V_{F\otimes \psi}(\xi,\mu)|^2\,d\xi \, d\mu .
\end{align}
In the following, we write
\[
(2(\textbf{k}+\textbf{s})+\textbf{r}) \cdot\textbf{b}^\mu = \sum_{n=1}^N \left(2(k_n+s_n)+r_n\right) b_n^\mu.
\]
Then, using \cref{lem:GAFA-adapted}, the right-hand side of \cref{eq:Plancherel-V} can be bounded by
\begin{align}
\sum_{\iota \in I_{\alpha}} \int_{R_\iota} \sum_{\mathbf{k}\in \N^N} \int_{\g_{2,r}^*} |\mu|^{-2\alpha} |D^{\delta_\iota} \psi(\mu)|^2  \big| F^{(\gamma_\iota)}\big((2(\textbf{k}+\textbf{s})+\textbf{r}) \cdot\textbf{b}^\mu\big)\big|^2 \notag \\
\times \bigg| \prod_{n=1}^N (b^\mu_n)^{1+a^\iota_n} (k_n+1)^{a^\iota_n} \bigg| \,d\nu_{\iota}(s) \, d\mu . \label{eq:GAFA-rewritten}
\end{align}
Note that $|\partial_v^{\delta_\iota}\psi(\mu)|\lesssim 2^{\ell\delta_\iota}$ for $\mu\in \supp \psi$. Thus, 
\[
|D^{\delta_\iota}\psi(\mu)|\lesssim 1 \quad \text{for } \mu\in \supp \psi.
\]
Recall that the assumptions of \cref{prop:weighted-plancherel-2} require that the multiplier $F$ is supported in the fixed compact subset $A\subseteq (0,\infty)$. Thus,
\[
|{\left( k_n + 1\right)b_n^\mu} |^{a_n^\iota}\lesssim_A 1
\quad \text{if } (2(\textbf{k}+\textbf{s})+\textbf{r}) \cdot\textbf{b}^\mu  \in \supp F.
\]
Hence, \cref{eq:GAFA-rewritten} is bounded by a constant times
\begin{align}
\sum_{\iota \in I_{\alpha}} \int_{R_\iota} & \sum_{\textbf{k} \in \N^N} \int_{\g_{2,r}^*} \left|\mu\right|^{-2\alpha} \tilde\chi(2^\ell |\mu|) \notag  \\
& \times \big|F^{(\gamma_\iota)}\big( (2(\textbf{k}+\textbf{s})+\textbf{r}) \cdot\textbf{b}^\mu \big)\big|^2 \, \prod_{n=1}^N b_n^\mu  \,d\mu \,d\nu_{\iota}(\textbf{s}), \label{eq:GAFA-rewritten-ii}
\end{align}
where $\tilde\chi$ is some non-negative bump function on $(0,\infty)$. We rewrite the integral over $\mu$ in polar coordinates, that is,
\[
\mu = \rho \omega \quad\text{for } \rho \in [0,\infty) \text{ and } |\omega|=1.
\]
Then, since $\mu\mapsto\textbf{b}^\mu$ is homogeneous of degree 1,
\[
\eigvp{k}{\mu} = \rho \eigvp{k}{\omega}.
\]
Thus, using $\rho=|\mu|\sim 2^{-\ell}$, \cref{eq:GAFA-rewritten-ii} is bounded by a constant times
\begin{align*}
2^{\ell (2\alpha -N -d_2 )}\sum_{\iota \in I_{\alpha}} & \int_{R_\iota}  \sum_{\textbf{k} \in \N^N} \int_{S^{d_2-1}} \int_0^\infty \tilde\chi(2^\ell \rho)  \\
& \times \big|F^{(\gamma_\iota)}\big( (2(\textbf{k}+\textbf{s})+\textbf{r}) \cdot\textbf{b}^\omega \rho \big)\big|^2 \, \prod_{n=1}^N b_n^\omega \,\frac{d\rho}{\rho} \,d\sigma(\omega) \,d\nu_{\iota}(\textbf{s}).
\end{align*}
The substitution $\lambda=(2(\textbf{k}+\textbf{s})+\textbf{r})\cdot\textbf{b}^\omega\rho$ yields that above expression equals
\begin{align*}
2^{\ell(2\alpha-N-d_2)}
\sum_{\iota \in I_{\alpha}} & \int_{R_\iota} \sum_{\textbf{k} \in \N^N} \int_{S^{d_2-1}} \int_0^\infty  \tilde\chi\bigg(\frac{2^\ell \lambda}{(2(\textbf{k}+\textbf{s})+\textbf{r}) \cdot\textbf{b}^\omega }\bigg)\\
& \times |F^{(\gamma_\iota)}(\lambda)|^2 
\prod_{n=1}^N b^\omega_n  \,\frac{d\lambda}{\lambda} \,d\sigma(\omega) \,d\nu_{\iota}(\textbf{s}).
\end{align*}
In the above sum in $\textbf{k}$, the number of non-vanishing summands is bounded by a constant times $(2^\ell \lambda)^N \prod_{n=1}^N (b^\omega_n)^{-1}$. Hence, using $\lambda\lesssim_A 1$ if $\lambda\in \supp F$, the last term can be bounded by
\[
2^{\ell(2\alpha-d_2)} \sum_{\iota \in I_{\alpha}} \int_0^\infty  
|F^{(\gamma_\iota)}(\lambda )|^2 \,d\lambda,
\]
which can be bounded by $C_\alpha 2^{\ell(2\alpha-d_2)}\norm{F}_{L^2_\alpha}^2$.
\end{proof}

To conclude the proof of the weighted Plancherel estimates \cref{eq:weighted-plancherel-2}, it remains to show the bounds in \eqref{eq:mu-bounds}. Following \cite{MaMue14b}, the arguments for these bounds rely heavily on the fact that the Lie algebra $\so_4$ of skew-symmetric real $4\times 4$ matrices admits a decomposition into two simple ideals isomorphic to $\su_2$ to write down the corresponding eigenvalues and spectral projections of $-J_\mu^2$.

\begin{proposition}
There is a direction $v\in\g_2$ with $|v|=1$ such that the bounds \cref{eq:mu-bounds} hold for the vector field $D=|\mu|\,\partial_v$.
\end{proposition}

\begin{proof}
Given $\mu\in\g_2^*$, let again $J_\mu$ denote the skew-symmetric matrix associated with the bilinear form $\omega_\mu(x,x')=\mu([x,x'])$ on $\g_1$, that is,
\[
\langle J_\mu x,x'\rangle=\omega_\mu(x,x') \quad\text{for all } x,x'\in\g_1,
\]
where $\langle\cdot,\cdot\rangle$ denotes the inner product rendering the basis $X_1,\dots,X_{d_1},U_1,\dots,U_{d_2}$ an orthonormal basis of $\g$. The matrices $J_\mu$ give rise to an embedding $\g_2^*\to \so_4$ into the Lie algebra $\so_4$ of skew-symmetric real $4\times 4$ matrices via $\mu\mapsto J_\mu$, whence $\g_2^*$ may be identified with a three-dimensional subspace $V$ of $\so_4$ (which corresponds to the fact that every Lie algebra emerges as a factor along the center of a free Lie algebra). By identifying $\R^4$ with $\C^2$, we may identify the Lie algebra $\su_2$ of skew-Hermitian complex $2\times 2$ matrices with a subspace of $\so_4$, which is the set of matrices of the form
\[
J^-(\xi) := 
\begin{pmatrix}
0 & -\xi_3 & -\xi_1 & -\xi_2 \\
\xi_3 & 0 & \xi_2 & -\xi_1 \\
\xi_1 & -\xi_2 & 0 & \xi_3 \\
\xi_2 & \xi_1 & -\xi_3 & 0 \\
\end{pmatrix},
\quad \xi = (\xi_1,\xi_2,\xi_3) \in \R^3.
\]
Conjugating $\su_2$ by the $4\times 4$ diagonal matrix $K=\diag(1,1,1,-1)$ yields the Lie algebra $\widetilde{\su_2}=K\su_2 K$ consisting of the matrices
\[
J^+(\eta) := 
\begin{pmatrix}
0 & -\eta_3 & -\eta_1 & \eta_2 \\
\eta_3 & 0 & \eta_2 & \eta_1 \\
\eta_1 & -\eta_2 & 0 & -\eta_3 \\
-\eta_2 & -\eta_1 & \eta_3 & 0 \\
\end{pmatrix},
\quad \eta = (\eta_1,\eta_2,\eta_3) \in \R^3,
\]
which decomposes $\so_4$ into the simple ideals $\su_2$ and $\widetilde{\su_2}$. Using the coordinates $\xi,\eta$ we may identify $\su_2$ and $\widetilde\su_2$ each with $\R^3$, that is,
\[
\so_4 \cong \su_2 \oplus \widetilde{\su_2} \cong R^3_\xi \oplus \R^3_\eta.
\]
We endow $\so_4$ with the inner product $\langle J,J' \rangle = -\tr(JJ')$, $J,J'\in\so_4$. Hence, if we let $J(\xi,\eta)=J^-(\xi)+J^+(\eta)$, then
\begin{equation}\label{eq:HS-norm}
\langle J(\xi,\eta),J(\xi',\eta') \rangle = \langle \xi,\xi' \rangle_{\R^3}+\langle \eta,\eta' \rangle_{\R^3}.
\end{equation}
By Proposition 25 of \cite{MaMue14b}, the two eigenvalues of $-J(\xi,\eta)^2$ are given by\[
b_1(\xi,\eta) = \left|\xi\right|+\left|\eta\right|
\quad\text{and}\quad
b_2(\xi,\eta) = \left|\left|\xi\right|-\left|\eta\right|\right|.
\]
If both $\xi$ and $\eta$ are non-zero, the spectral projections onto the corresponding eigenspaces are given by
\begin{equation}\label{eq:projections}
\begin{aligned}
P_1(\xi,\eta) & = \frac 1 2 \id_{\R^4} - \frac 1 2 \frac{J^+(\eta)}{|\eta|} \frac{J^-(\xi)}{|\xi|},\\
P_2(\xi,\eta) & = \frac 1 2 \id_{\R^4} + \frac 1 2 \frac{J^+(\eta)}{|\eta|} \frac{J^-(\xi)}{|\xi|}.
\end{aligned}
\end{equation}
We identify the three-dimensional subspace $V\subseteq \so_4$ by a subspace of $\R^3_\xi\oplus R^3_\eta$ via the coordinates $\xi,\eta$. Since $G$ is assumed to be a Métiver group, the matrices $J(\xi,\eta)$ have to be invertible for all $(\xi,\eta)\in V\setminus\{0\}$. Thus, $|\xi|\neq |\eta|$ for all $(\xi,\eta)\in V\setminus\{0\}$, since the eigenvalue $b_2(\xi,\eta)$ would vanish otherwise. This in turn implies
\[
V\cap (\{0\}\oplus \R_\eta^3)=\{0\}
\quad\text{or}\quad
V\cap (\R_\xi^3\oplus\{0\})=\{0\}.
\]
(Otherwise, we could find elements $(0,\eta)\in V$ and $(\xi,0)\in V$ with $\xi,\eta\neq 0$ such that $|\xi|=|\eta|$.) Hence $V$ can be written as $V = \{ (\xi,A\xi) :\xi\in\R^3 \}$ or  $V = \{ (A\eta,\eta) :\eta\in\R^3 \}$ for some $3\times 3$ matrix $A$. In the former case (the latter case is similar), the eigenvalues of $-J(\xi)^2:=-J(\xi,A\xi)^2$ are given by
\begin{equation}\label{eq:xi-eigenvalues}
\begin{aligned}
b_1(\xi) & = b_1(\xi,A\xi)
  = \left|\xi\right|+\left|A\xi\right|,\\
b_2(\xi) & = b_2(\xi,A\xi)
  = \left|\left|\xi\right|-\left|A\xi\right|\right|.
\end{aligned}
\end{equation}
We distinguish the cases $\ker A=\{0\}$, $\ker A=\R^3$ and $\{0\}\neq \ker A\neq \R^3$. In the following, we identify points $\mu\in\g_2^*$ with points $\xi\in\R^3$ by the chosen coordinates.

If $\ker A=\{0\}$, we have the decomposition
\[
-J(\xi)^2 = b_1(\xi)^2 P_1(\xi) + b_2(\xi)^2 P_2(\xi) \quad \text{for all } \xi\neq 0,
\]
and both the eigenvalues $b_j(\xi)$ and spectral projections $P_j(\xi)=P_j(\xi,A\xi)$ in \cref{eq:projections} are smooth functions of $\xi$ away from the origin $\xi=0$. Note that $\xi\mapsto b_j(\xi)$ and $\xi\mapsto P_j(\xi)$ are homogeneous of degrees 1 and 0, respectively. Since $b_j(\xi)$ vanishes only at the origin, the desired bounds \cref{eq:mu-bounds} are a consequence of homogeneity and the fact that all derivatives of $\xi\mapsto b_j(\xi)$ and $\xi\mapsto P_j(\xi)$ are smooth functions on the 2-sphere in $\R^3_\xi$, and we may choose the vector $v\in\g_2$ with $|v|=1$ of the directional derivative $\partial_v$ arbitrarily.

If $\ker A=\R^3$, then \cref{eq:xi-eigenvalues} yields $b_1(\xi)=b_2(\xi)=|\xi|$, and $-J_\mu^2$ is just a multiple of the identity, whence $G$ is a even a Heisenberg type group in that case. Both the eigenvalue and the spectral projection (being the identity) are smooth away from the origin, and as in the previous case, we obtain \cref{eq:mu-bounds} by choosing an arbitrary direction $v\in\g_2$ with $|v|=1$.

Now suppose that $\{0\}\neq\ker A\neq\R^3$. Then we have the decomposition
\[
-J(\xi)^2 = b_1(\xi)^2 P_1(\xi) + b_2(\xi)^2 P_2(\xi) \quad \text{for all } \xi\notin \ker A.
\]
Note that the set $\g_{2,r}^*=\{\xi\in\R^3:\xi\notin \ker A\}$ is in particular a Zariski-open subset of $\g_2^*$. Both functions $\xi\mapsto b_j(\xi)$ and $\xi\mapsto P_j(\xi)$ are smooth on $\g_{2,r}^*$ and continuous on $\g_2$, but their directional derivatives may admit singularities along $\g_2^*\setminus\g_{2,r}^*$ if the direction of the derivative intersects this set transversely. However, thanks to fact that the Zariski-closed set $\g_2^*\setminus\g_{2,r}^*$ is either a plane or a line, choosing $v\in\ker A\setminus\{0\}$ with $|v|=1$, every line pointing in the same direction as $v$ either lies completely in the Zariski-closed set $\g_2^*\setminus\g_{2,r}^*$ or has no intersection with it, meaning that $\xi\mapsto b_j(\xi)$ and $\xi\mapsto P_j(\xi)$ are smooth away from the origin if we only take derivatives along lines with direction $v$.

More precisely, for $v\in\ker A\setminus\{0\}$, since $A(\xi+tv)=A\xi$ for $t\in\R$ then, the first derivatives of $b_1(\xi)$ and $b_2(\xi)$ are given by
\[
\partial_v b_1(\xi) = \partial_v(|\xi|)
\quad\text{and}\quad
\partial_v b_2(\xi) = \sgn(|\xi|-|A\xi|)\, \partial_v(|\xi|),
\]
which are both smooth functions away from the origin since $b_2(\xi)=\left|\left|\xi\right|-\left|A\xi\right|\right|\neq 0$ for $\xi\neq 0$. Similarly, for the projections $P_j(\xi)$, we observe
\[
\partial_v \left(\frac{J^+(A\xi)}{|A\xi|} \frac{J^-(\xi)}{|\xi|}\right)
 = \frac{J^+(A\xi)}{|A\xi|} \,\partial_v \frac{J^-(\xi)}{|\xi|},
\]
meaning that $v$-derivatives only hit the second factor $J^-(\xi)/|\xi|$ in \cref{eq:projections}, which is smooth away from the origin. On the other hand, due to \cref{eq:HS-norm},
\[
\left| \frac{J^+(A\xi)}{|A\xi|} \right|\le 1.
\]
Thus, arguing again via smoothness and homogeneity for the factor $J^-(\xi)/|\xi|$, we obtain the desired bounds \cref{eq:mu-bounds}.
\end{proof}


\end{document}